\documentclass[12pt,reqno,oneside]{amsart}
\usepackage{amsmath}
\usepackage{amssymb}
\usepackage{amsthm}
\usepackage{graphics}
\usepackage{eucal}
\usepackage{mathrsfs}
\usepackage{color}
\usepackage[pdftex,bookmarks,colorlinks,breaklinks]{hyperref}
\usepackage{mathtools}
\newtheorem{theorem}{Theorem}[section]
\newtheorem{definition}{Definition}[section]
\newtheorem{proposition}{Proposition}[section]
\newtheorem{lemma}{Lemma}[section]
\newtheorem{corollary}{Corollary}[section]

\numberwithin{equation}{section}

\newcommand{\be}{\mathbf{e}}

\newcommand{\diag}{\operatorname{diag}}

\newcommand{\GL}{\operatorname{GL}}

\newcommand{\supp}{\operatorname{supp}}
\newcommand{\cov}{\operatorname{cov}}

\newcommand{\bq}{\mathbf{q}}
\newcommand{\bx}{\mathbf{x}}

\newcommand{\cL}{\mathcal{L}}

\newcommand{\bw}{\mathbf{w}}
\newcommand{\f}{\mathbf{f}}
\newcommand{\rk}{\operatorname{rank}}

\newcommand{\ba}{\mathbf{a}}
\newcommand{\Id}{\operatorname{Id}}

\newcommand{\bc}{\mathbf{c}}

\newcommand{\by}{\mathbf{y}}

\newcommand{\al}{\alpha}

\newcommand{\bg}{\mathbf{g}}

\newcommand{\Q}{\mathbb{Q}}
\newcommand{\Z}{\mathbb{Z}}
\newcommand{\R}{\mathbb{R}}
\newcommand{\N}{\mathbb{N}}

\newcommand{\inv}{^{\text{-}1}}

\usepackage{mathtools}

\begin{document}

\title[$p$-adic Diophantine approximation]{Diophantine Inheritance for $p$-adic measures}

\author{Shreyasi Datta}
\address{\textbf{Shreyasi Datta} \\
School of Mathematics,
Tata Institute of Fundamental Research, Mumbai, India 400005}
\email{shreya@math.tifr.res.in}

\author{Anish Ghosh}
\address{\textbf{Anish Ghosh} \\
School of Mathematics,
Tata Institute of Fundamental Research, Mumbai, India 400005}
\email{ghosh@math.tifr.res.in}
\date{}

\thanks{A.\ G.\ gratefully acknowledges support from a grant from the Indo-French Centre for the Promotion of Advanced Research; a Department of Science and Technology, Government of India Swarnajayanti fellowship and a MATRICS grant from the Science and Engineering Research Board.}

\date{}

\begin{abstract}
In this paper we prove complete $p$-adic analogues of Kleinbock's theorems \cite{Kleinbock-extremal, Kleinbock-exponent} on inheritance of Diophantine exponents for affine subspaces. In particular, we answer in the affirmative (and in a stronger form), a conjecture of Kleinbock and Tomanov \cite{KT}, as well as a question of Kleinbock \cite{Kleinbock-exponent}. Our main innovation is the introduction of a new $p$-adic Diophantine exponent which is better suited to homogeneous dynamics, and which we show to be closely related to the exponent considered by Kleinbock and Tomanov. 
\end{abstract}

\maketitle

\section{Introduction}\label{sec:intro}
This paper is concerned with the study of $p$-adic Diophantine approximation on manifolds. We briefly recall the setting and basic results from the paper \cite{KT} of Kleinbock and Tomanov.\\

For $\bq = (q_1, \dots, q_n) \in \Z^n$ and $q_0 \in \Z$, set  $\tilde{\bq} := (q_0, q_1, \dots, q_n)$. Define the Diophantine exponent $w(\by)$ of $\by \in \Q_p^{n}$ to be the supremum of $v > 0$ such that there are infinitely many $\tilde{\bq} \in \Z^{n+1}$ such that
\begin{equation}\label{def:exp}
|q_0 + \bq\cdot \by|_p  \leq \|\tilde{\bq}\|_\infty^{-v}.
\end{equation}
In view of Dirichlet's theorem (\cite{KT} \S 11.2), $w(\by) \geq n+1$ for every $\by \in \Q_p^{n},$ with equality for Haar almost every $\by$ by the Borel-Cantelli lemma. 

A vector $\by \in \Q^{n}_p$ is called very well approximable (VWA) if $w(\by) > n + 1$. The set of very well approximable vectors has zero Haar measure. Diophantine approximation on manifolds or ``Diophantine approximation with dependent quantities" is concerned with the question of \emph{inheritance} of Diophantine properties which are generic with respect to Lebesgue measure in $\R^n$ or $\Q_{p}^{n}$ with respect to Lebesgue measure, by appropriate proper subsets. The theory began with a conjecture of Mahler, that almost every point on the Veronese curve 
$$ (x, x^2, \dots, x^n) \subset \R^n $$
is not very well approximable. Mahler made his conjecture in the context of a classification of numbers in terms of their approximation properties. We refer the reader to Bugeaud's book \cite{Bugeaud-book} for details about developments in this area. Mahler's conjecture was settled by Sprind\v{z}huk, in the real, $p$-adic and positive characteristic (i.e. function fields in one variable over a finite field) contexts. It was conjectured by Sprind\v{z}uk in 1980 \cite{Sp3, Sp2} and proved by Kleinbock and Margulis \cite{KM} that real analytic manifolds not contained in any proper affine subspace of $\R^n$ are extremal. They introduced methods from the ergodic theory of group actions into the subject and \cite{KM} has become a cornerstone of the subject. A further breakthrough was achieved by Kleinbock, Lindenstrauss and Weiss in \cite{KLW}, where a wide class of measures, including fractal measures as well as pushforwards of Lebesgue measure by nondegenerate maps which were previously introduced by Kleinbock and Margulis (the definition follows in the next paragraph) were studied in a unified manner. Subsequently, Kleinbock and Tomanov \cite{KT} proved an $S$-adic version of Sprind\v{z}huk's conjectures and in \cite{Gpos}, the second named author proved the positive characteristic version of Sprind\v{z}huk's conjecture. In two important papers, \cite{Kleinbock-extremal, Kleinbock-exponent}, D. Kleinbock systematically explored Diophantine approximation on affine subspaces and their nondegenerate submanifolds and the second named author proved Khintchine type theorems both convergence and divergence \cite{G1, G-thesis, G-div, G-mult, G-monat}. Inhomogeneous and quantitative versions of Khintchine's theorem for affine subspaces have been proved in \cite{BGGV, GG}. We refer the reader to the survey \cite{G-handbook} of the second named author for more details on the problem of Diophantine approximation on affine subspaces.\\

The subject of $p$-adic Diophantine approximation started with the work of E. Lutz \cite{Lutz} and has seen numerous advances in different contexts over the years including early work of Mahler \cite{Mah2, Mah}. We refer the reader to \cite{Bugeaud-book} for a comprehensive reference, to \cite{KT} and the references therein for work relating to $p$-adic metric Diophantine approximation on manifolds, and to \cite{BK, BBK, BBDO, DG} for recent results.\\

In this paper, we provide a complete $p$-adic analogue of the results of D. Kleinbock  \cite{Kleinbock-exponent} on Diophantine exponents of affine subspaces in two types of approximations, one was considered by Kleinbock and Tomanov in \cite{KT} and the other one is a new type of approximation we introduced here. This answers, in a stronger form, Conjecture IS of Kleinbock and Tomanov \cite{KT} as well as a question of D. Kleinbock \cite{Kleinbock-exponent}.\\

While the methods of the present paper are heavily influenced by the work of D. Kleinbock, providing a complete $p$-adic analogue of his results poses substantial new difficulties. To circumvent these difficulties, one of the key innovations of the present paper is a new $p$-adic Diophantine exponent which measures approximation with respect to $\mathbb{Z}[1/p]$-points.

We now introduce some notation and state the main results in the present paper. 
For a Borel measure $\mu$ on $\Q_p^{n}$, we follow \cite{Kleinbock-exponent} in defining the Diophantine exponent $\omega(\mu)$ of $\mu$ to be 
\begin{equation}
w(\mu) = \sup\{ v~:~ \mu(\{y~|~w(y) > v\})>0\}.
\end{equation}

The exponent only depends on the measure class of $\mu$. Let $\lambda$ denote Haar measure on $\Q^{n}_{p}$ normalized so that $\Z^{n}_p$ has volume $1$, the dimension being clear from the context. If $\mathcal{M} \subset \Q_p^{n}$ is a $d$ dimensional analytic manifold, and $\mu$ is the pushforward of Haar measure on $\Q_p^{d}$ by a map parametrising $\mathcal{M}$, then $\omega(\mathcal{M})$ is defined to be $\omega(\mu)$. Following Kleinbock \cite{Kleinbock-exponent}, we say that a differentiable map $f : U \to \Q^{n}_{p}$, where $U$ is an open subset of $\Q^{d}_{p}$, is nondegenerate in an affine subspace $\mathcal{L}$ of $\Q^{n}_p$ at $x \in U$ if $f(U) \subset \mathcal{L}$ and the span of all the partial derivatives of $f$ at $x$ up to some order coincides with the linear part of $\mathcal{L}$. If $\mathcal{M}$ is a $d$-dimensional submanifold of $\mathcal{L}$, we will say that $\mathcal{M}$ is nondegenerate in $\mathcal{L}$ at $y \in \mathcal{M}$ if there exists a diffeomorphism $f$ between an open subset $U$ of $\Q^{d}_{p}$ and a neighbourhood of $y$ in $\mathcal{M}$ is nondegenerate in $\mathcal{L}$ at $f^{-1}(y)$. Finally, we will say that $f:U\to \mathcal{L}$ (resp., $\mathcal{M} \subset \mathcal{L}$) is nondegenerate in $\mathcal{L}$ if it is nondegenerate in $\mathcal{L}$ at $\lambda$-a.e. point of $U$ (resp., of $\mathcal{M}$, in the sense of the smooth measure class on $\mathcal{M}$). Here is a special case of one of our main results.

\begin{theorem}\label{main:sp}
Let $\mathcal{L}$ be an affine subspace of $\Q_p^n$, and let $\mathcal{M}$ be a submanifold of $\mathcal{L}$ which is nondegenerate in $\mathcal{L}$. Then 
\begin{equation}
w(\mathcal{M})=w(\mathcal{L})= \inf\{ w(\by)~|~\by\in\mathcal{L}\}= \inf\{w(\by) ~|~ \by \in \mathcal{M}\}.
\end{equation}
\end{theorem}

\noindent In particular, this implies that if $\cL$ is an extremal subspace of $\Q_p^n$, and $\mathcal{M}$ is nondegenerate in $\cL$, then $\mathcal{M}$ is extremal. As noted by Kleinbock in the context of real Diophantine approximation, the middle equality is non-obvious and non-trivial. The difficulty of proving it persists in the $p$-adic context and is in fact compounded by the difficulty, in the $p$-adic setting, of translating the Diophantine problem to dynamics and back. To overcome this challenge, we introduce the new idea of $\Z[1/p]$-exponents. We define these exponents $w_p(\by)$ and $w_{p}(A)$ in the next sections; the point is that these exponents seem to be better suited to techniques from homogeneous dynamics and in particular allow us to prove an ``if and only if" Dani-type correspondence in the $p$-adic setting, which does not seem to be readily achieved whilst dealing with approximation of $p$-adic vectors by rational numbers. We suspect these exponents will find further use. We then explain how the $\Z[1/p]$-exponent is related to the usual exponent. We  also define $p$-adic exponents $\omega_p(\mu)$ of a borel measure $\mu$ on $\Q_p^n$ to be 
\begin{equation}
w_p(\mu) := \sup\{ v~:~ \mu(\{y~|~w_p(y) > v\})>0\}
\end{equation}\\ and prove the following theorem.
\begin{theorem}\label{main:sp2}
	Let $\mathcal{L}$ be an affine subspace of $\Q_p^n$, and let $\mathcal{M}$ be a submanifold of $\mathcal{L}$ which is nondegenerate in $\mathcal{L}$. Then 
	\begin{equation}
	w_p(\mathcal{M})=w_p(\mathcal{L})= \inf\{ w_p(\by)~|~\by\in\mathcal{L}\}= \inf\{w_p(\by) ~|~ \by \in \mathcal{M}\}.
	\end{equation}
\end{theorem}

We will also provide a condition for $p$-adic extremality in terms of Diophantine properties of the parametrising matrix of an affine subspace. 

\begin{theorem}
	Let $\mathcal{L}$ be an affine subspace, parametrized by a matrix $R_A$ as in (\ref{A}). If all the rows (resp. columns) are rational multiples of one row (resp. column) then one has $$
w_{p}(\mathcal{L})= \max(n, w_{p}(A)) \text{ and } w(\mathcal{L})=\max(n+1, w(A)).
$$
\end{theorem}
 
Diophantine approximation on affine subspaces plays in role in KAM theory, see \cite{Sevryuk, SuRafa} and we hope that the present paper might also find applications. 
The setup of Diophantine approximation by $\mathbb{Z}[1/p]$-points has already been considered in considerable detail in the context of \emph{intrinsic Diophantine approximation} on varieties by Ghosh, Gorodnik and Nevo \cite{GGN1, GGN2}.   However, as we have mentioned, as far as we are aware the present paper is the first one to consider it in the context of (extrinsic) Diophantine approximation on manifolds.  Moreover, we are not aware of any such connection between exponents of $\Z$-approximation and $\Z[\frac{1}{p}]$-approximation in case of reals. 
\subsection*{Structure of the paper}
The main results of the paper are proved in the more general setting of Federer measures and nonplanar maps. In the next section we present these and more definitions leading up to the statement of Theorem \ref{main:fed} from which Theorem \ref{main:sp} follows. Section \ref{sec:expo} introduces a $p$-adic exponent $w_{p}$ and discusses the relationship between $w_{p}$ and $w$. In \S \ref{dani} we prove an important dictionary in $p$-adic dynamics: namely the explicit connection between $p$-adic exponents and homogeneous dynamics. In \S \ref{margulis} we extend the quantitative nondivergence theorem proved by Kleinbock to the $p$-adic setting. The subsequent section \ref{kleinbock} deals with applications on nondivergence to $p$-adic exponents, in particular proving Theorem \ref{main:fed}. The final section \ref{higher} deals with higher Diophantine exponents.

\subsection*{Acknowledgements} Part of this work was done when both authors were visiting the Weizmann institute. We thank Uri Bader and the Weizmann institute for the hospitality. A. G. gratefully acknowledges the support of the Benoziyo Endowment Fund for the Advancement of Science at the
Weizmann Institute. Another part of this work was done when the authors were visiting the ICTS, Bengaluru. We thank the institute for the pleasant working conditions.

\section{Preliminaries}
\subsection*{Measures and spaces}
A metric space $X$ is called \emph{Besicovitch} \cite{KT} if there exists a constant $N_X$ such that the following holds: for any bounded subset $A$ of $X$ and for any family $\mathcal{B}$ of nonempty open balls in $X$ such that
$$ \forall x \in A \text{ is a center of some ball of } B,$$
 there is a finite or countable subfamily $\{B_i\}$ of $B$ with
 $$ 1_A \leq \sum_{i}1_{B_i} \leq N_X. $$

As remarked in \cite{KT}, any separable ultrametric space $X$ is Besicovitch with $N_X = 1$. We now define $D$-Federer measures following \cite{KLW}. Let $\mu$ be a Radon measure on $X$, and $U$ an open subset of $X$ with $\mu(U) > 0$. We  say that $\mu$ is $D$-Federer on $U$ if
$$ \sup_{\substack{x \in \supp \mu, r > 0\\ B(x, 3r) \subset U}} \frac{\mu(B(x, 3r))}{\mu(x,r)} < D.$$
Finally, we say that $\mu$ as above is Federer if for $\mu$-a.e. $x \in X$ there exists a neighbourhood $U$ of $x$ and $D > 0$ such that $\mu$ is $D$-Federer on $U$. We refer the reader to \cite{KLW, KT} for examples of Federer measures.\\

Following, \cite{Kleinbock-exponent}, for a subset $M$ of $\Q_{p}^n$, define its affine span $\langle M\rangle_a$ to be the intersection of all affine subspaces of $\Q_{p}^n$ containing $M$. Let $X$ be a metric space, $\mu$ a measure on $X$, $\cL$ an affine subspace of  $\Q_{p}^n$ and $f$ a map from $X$ into $\cL$. Say that $(f, \mu)$ is nonplanar in $\cL$ if
$$\cL = \langle f (B \cap \supp \mu)\rangle_a \forall \text{ nonempty open } B \text{ with } \mu(B) > 0.$$

\subsection*{$(C, \alpha)$-good functions}\label{sec:good}
In this section, we recall the notion of $(C, \alpha)$-good functions on ultrametric spaces. We follow the treatment of Kleinbock and Tomanov \cite{KT}. Let $X$ be a metric space, $\mu$ a Borel measure on $X$ and let $(F, |\cdot|)$ be a valued field. For a subset $U$ of $X$ and $C, \alpha > 0$, say that a Borel measurable function $f : U \to F$ is $(C, \alpha)$-good on $U$ with respect to $\mu$ if for any open ball $B \subset U$ centered in $\supp \mu$ and $\varepsilon > 0$ one has
\begin{equation}\label{gooddef}
\mu \left(\{ x \in B \big| |f(x)| < \varepsilon \} \right) \leq
C\left(\displaystyle \frac{\varepsilon}{\sup_{x \in
B}|f(x)|}\right)^{\al}|B|.
\end{equation}

Where $\|f\|_{\mu, B} = \sup \{c : \mu(\{x \in B : |f(x)| > c\}) > 0\}$.

\begin{theorem}\label{good}[Theorem $4.3$ \cite{KT}]
 Let $F$ be either $\R$ or an ultrametric valued field, and let $f$ be a $C^l$ map from an open subset $U$ of $F^d$ to $F^n$. Then $f$ is nonplanar and good at every point of $U$ where it is nondegenerate.        
\end{theorem}

As a corollary, we have 
\begin{corollary}\label{affgood}
Let $\mathcal{L}$ be an affine subspace of $\Q_{p}^n$ and let $f = (f_1,\dots, f_n)$ be a smooth map from an open subset $U$ of $\Q_{p}^d$ to $\mathcal{L}$ which is nondegenerate in $\mathcal{L}$ at $x_0 \in U$. Then $f$ is good at $x_0$.
\end{corollary}            
               
\begin{proof}
This follows from Theorem \ref{good}, see Corollary $3.2$ in \cite{Kleinbock-extremal}.
\end{proof}

\subsection*{Main Theorem}

Our main result is a complete $p$-adic  analogue of Theorem $0.3$ of \cite{Kleinbock-exponent}.

\begin{theorem}\label{main:fed}
	Let $\mu$ be a Federer measure on a Besicovitch metric space $X, \mathcal{L}$ an affine subspace of $\Q_p^n$, and let $\f : X\to \mathcal{L}$ be a continuous map such that $(\f, \mu)$ is good and nonplanar in $\mathcal{L}$. Then 
	\begin{equation}
	w(\f_*\mu)=w(\mathcal{L})= \inf\{ w(\by)~|~\by\in\mathcal{L}\}= \inf\{w(\f(x)) ~|~ x\in\supp\mu\}.
	\end{equation}
\end{theorem}

\section{$p$-adic Diophantine exponents}\label{sec:expo}
We begin with some motivation for $p$-adic Diophantine approximation and the definition of $v$-approximable numbers in $\Q_p^{n},$ both using $\Z$ and $\Z[1/p]$ approximations. A natural starting point should be an analogue of Dirichlet's theorem in this set up. In \cite{KT} a $p$-adic Dirichlet's theorem using $\Z$ approximations has been discussed in detail. Here we observe that Dirichlet's theorem using $\Z$ approximations does indeed give a Dirichlet theorem in case of $\Z[1/p]$ approximations. \\

 Here and below, we adopt the notation $\tilde{\bq} :=(q_0,\bq)$ for $q_0, \bq$ in $\Z$ as well as $\Z[1/p]$. Lets recall the Dirichlet's theorem in \cite{KT}, for $\by\in\Q_p^n$ and for every $Q>0$ there exists an integer solution for the system 
\begin{align}
\vert q_0+\bq.\by\vert_p & \leq \frac{const(\by)}{Q},\\ 
\Vert \tilde \bq\Vert^{n+1}_{\infty} & \leq Q.
\end{align}
Now observe that this implies that for every $Q>0$ there exists an integer solution to the system 
\begin{align}
\Vert\tilde{\bq}\Vert_\infty.\vert q_0+\bq.\by\vert_p &\leq \frac{const(\by)}{Q^n},\\
\Vert\tilde{\bq}\Vert_{\infty} &\leq Q.
\end{align}

\noindent Note that $\Vert{\bq}\Vert_p\leq 1$ for integer vectors, so we have that for any $\by\in\Q_p^n$ and every $Q > 0$ there exists a $\Z[1/p]^{n+1} $ solution to the system 
\begin{align}
\Vert\tilde{\bq}\Vert_\infty.\vert q_0+\bq.\by\vert_p &\leq \frac{const(\by)}{Q^n},\\
\Vert\bq\Vert_p \Vert\tilde{\bq}\Vert_\infty
&\leq Q.
\end{align}
This motivates the follwoing definition.
\begin{definition}{\textbf{v-$\Z[1/p]$ approximable vectors:}} $\by\in\Q_p^n$ is $v-\Z[1/p]$-approximable if there exist $\tilde{\bq}=(q_0,\bq)\in\Z[1/p]^{n+1}$ with unbounded $||\bq||_p||\tilde{\bq}||_\infty$ such that 
		\begin{equation}\label{Z_S}
		|\bq.\by+q_0|_p<\frac{1}{(||\bq||_p||\tilde{\bq}||_\infty)^v||\tilde{\bq}||_\infty}.
		\end{equation}
	\end{definition}
\noindent
We will denote $v-\Z[1/p]$-approximable points by $\mathcal{W}_v^{p}$ and also define 
\begin{equation}\label{def:papp}
w_{p}(\by):=\sup\{v \text{ appearing in } (\ref{Z_S}) \}.
\end{equation}

Similarly, we will denote $v$-approximable points by $\mathcal{W}_v$ and recall that we have defined $$w(\by):=\sup\{v \text{ appearing in } (\ref{def:exp}) \}. $$

We therefore have two Diophantine exponents:
\begin{enumerate}
\item $\omega$, defined in (\ref{def:exp}) involving  $\Z$-Diophantine approximation of $p$-adic vectors. \\
\item $\omega_p$, defined in (\ref{def:papp}) involving $\Z[1/p]$-Diophantine approximation of $p$-adic vectors.
\end{enumerate}
Later on, we will need higher Diophantine exponents of both kinds. Although the two types of approximations are a priori different, we will shortly prove that the Diophantine exponents are very closely related.
The following lemma will come handy to compare these Diophantine exponents and further to relate Diophantine approximation to dynamics.
\begin{lemma}\label{setproperty}
Consider the set \begin{equation}\label{theset}
E= \left\{ |q_0+\bq.\by|_p\Vert\tilde{\bq}\Vert_\infty,\Vert\bq\Vert_p\Vert\tilde{\bq}\Vert_\infty ~\big | ~ \tilde{\bq}=(q_0,\bq)\in\Z[1/p]^{n+1}\right\}.
	\end{equation} If $(x_k,z_k)\in E$ such that $z_k$ is bounded and $x_k\to 0$ then $x_k=0$ for all but finitely many $k$.
\end{lemma}
\begin{proof}
	Suppose $\Vert \bq_k\Vert_p\Vert\tilde{\bq}_{k}\Vert_\infty \leq M$ for some $M>0$ and 
	$$\vert q_{0k}+\bq_k.\by\vert_p\Vert \tilde{\bq}_k\Vert_\infty\rightarrow 0, \text{ as }k\to\infty,$$ 
	where $\tilde{\bq}_k=(q_{0k},\bq_k)=(q_{ik})_{i=0}^{n}\in\Z[1/p]^{n+1}$. Since $$
	\vert q_{0k}\vert_p\Vert\tilde\bq_k\Vert_\infty\leq \vert q_{0k}+\bq_k.\by\vert_p\Vert \tilde{\bq}_k\Vert_\infty+ \Vert \bq_k\Vert_p\Vert\tilde{\bq}_{k}\Vert_\infty\Vert\by\Vert_{p},$$ we can choose $M$ such that $\vert q_{0k}\vert_p\Vert\tilde\bq_k\Vert_\infty\leq M$ and $\Vert \bq_k\Vert_p\Vert\tilde{\bq}_{k}\Vert_\infty \leq M$ and therefore $\Vert \tilde\bq_k\Vert_p\Vert\tilde{\bq}_{k}\Vert_\infty \leq M$.\\ 
	Note that there are only finitely many $p$-free integers in $q_{ik},$ i.e. in $\tilde \bq_k=(q_{ik})_{i=0}^{n}=(p^{m_{ik}}z_{ik})_{i=0}^n$ where $p\nmid z_{ik},$  $\vert z_{ik}\vert_\infty$ is bounded. This follows from the fact that $$
	\vert q_{ik}\vert_p\vert q_{ik}\vert_\infty=p^{-m_{ik}}.p^{m_{ik}}\vert z_{ik}\vert_\infty\leq\Vert \tilde\bq_k\Vert_p\Vert\tilde{\bq}_{k}\Vert_\infty \leq M.$$ 
	So there are finitely many $z_{ik}$.\\
	We denote $\Vert \tilde \bq_k\Vert_\infty =p^{m_k}z_k$ where $p\nmid z_k, m_k\in\Z$. If $z_k=0$ then $\tilde \bq_k=0$. Otherwise 
	$$ p^{-m_{ik}}.p^{m_k}|z_k|_\infty \leq \Vert \tilde\bq_k\Vert_p\Vert\tilde{\bq}_{k}\Vert_\infty \leq M,$$ and similarly, 
	$$ p^{-m_{k}}.p^{m_{ik}}|z_{ik}|_\infty \leq \Vert \tilde\bq_k\Vert_p\Vert\tilde{\bq}_{k}\Vert_\infty \leq M.$$ 
	So for nonzero elements $q_{ik}=p^{m_{ik}}z_{ik}$ we have $\vert m_k-m_{ik}\vert$ bounded since we have already noted that there are finitely many $z_{ik}$. Therefore $\vert q_{0k}+\bq_k.\by\vert_p\Vert \tilde{\bq}_k\Vert_\infty$ has only finitely many options. So it can only go to $0$ if the terms are identically $0$ for all but finitely many possibilities.
\end{proof}
 Now we can conclude the following relations between exponents.
 \begin{proposition}\label{same}
 	For any $\by\in\Q_p^n$ we have $$w_{p}(\by)=w(\by)+1.$$
 		\end{proposition}
\begin{proof}
	Suppose $\by\in \mathcal{W}_v,$ then there exists infinitely many integer vectors $\tilde \bq=(q_0,\bq)\in \Z^{n+1}$ such that $$ 	|\bq.\by+q_0|_p<\frac{1}{(\Vert\tilde{\bq}\Vert_\infty)^v}.$$ Since $\Vert\bq\Vert_p\leq 1$ for  ${\bq}\in\Z^{n+1}$, the above inequality is the same as 
	$$ 	|\bq.\by+q_0|_p<\frac{1}{(\Vert\bq\Vert_p\Vert\tilde{\bq}\Vert_\infty)^{v-1}\Vert\tilde{\bq}\Vert_\infty}.$$ 
	Lemma \ref{setproperty} assures us that the $\Vert\bq\Vert_p\Vert\tilde{\bq}\Vert_\infty$ appearing here are unbounded when $v>1$. Hence, we have that $\mathcal{W}_v\subset \mathcal{W}_{v-1}^{p}$ when $v>1$. On the other hand, if $\by\in \mathcal{W}_v^{p}$ then there are unbounded many $\Vert\bq\Vert_p\Vert\tilde{\bq}\Vert_\infty$ such that
	\begin{equation}
	|\bq.\by+q_0|_p<\frac{1}{(||\bq||_p||\tilde{\bq}||_\infty)^v||\tilde{\bq}||_\infty},
	\end{equation}
	where $\tilde{\bq}=(q_0,\bq)\in\Z[1/p]^{n+1}$. This inequality can be rewritten as $$
		\vert\Vert\bq\Vert_p(\bq.\by+q_0)\vert_p<\frac{1}{(\Vert\bq\Vert_p\Vert\tilde{\bq}\Vert_\infty)^{v+1}}$$ for unbounded many $\Vert\bq\Vert_p\Vert\tilde{\bq}\Vert_\infty$ where $\tilde{\bq}=(q_0,\bq)\in\Z[1/p]^{n+1}$. But note that $\Vert\bq\Vert_p\bq\in \Z^{n+1}$. Hence the above criterion is the same as  $$
		\vert(\bq.\by+q_0)\vert_p<\frac{1}{(\Vert\tilde{\bq}\Vert_\infty)^{v+1}}$$ for unbounded many $\Vert\tilde{\bq}\Vert_\infty$ where $\tilde\bq=(q_0,\bq)\in\Z[1/p]\times\Z^n$. Now note that $\vert q_0\vert_p\leq\max(\Vert\by\Vert_p,1)$. So when $q_0\notin\Z$, we may multiply by $\vert q_0\vert_p>1 $ and since we have an upper bound for $\vert q_0\vert_p$ we get a $\Z$ approximation but for slightly smaller $v$. So we have for any $\varepsilon>0$
	$$
	\vert(\bq.\by+q_0)\vert_p<\frac{1}{(\Vert\tilde{\bq}\Vert_\infty)^{v-\varepsilon+1}}$$
	for unbounded many $\Vert\tilde{\bq}\Vert_\infty,$ and hence,  for infinitely many $\tilde{\bq}\in\Z^{n+1}$. Therefore $\mathcal{W}_v\subset \mathcal{W}_{v+1-\varepsilon}^{p}$. Hence the conclusion follows.
\end{proof}    
A quick observation which follows from Dirichlet's theorem and Proposition \ref{same} is that $w_{p}(\by)\geq n$ for all $\by\in\Q_p^n$.
We can define the Diophantine $\Z[1/p]$ exponent more generally for a matrix $A$ of order $m, n$ and we will need these notions later in the paper. Define 
\begin{equation}\label{def:matrixp}
w_{p}(A):=\sup\left\{v ~\biggl|\begin{aligned} &\text{ there exists unbounded  many} \Vert\bq\Vert_p\Vert\tilde{\bq}\Vert_\infty\\
& \text{ s.t. } \Vert A.\bq+\bq_0\Vert_p\leq \frac{1}{(\Vert\bq\Vert_p\Vert\tilde{\bq}\Vert_\infty)^v\Vert\tilde{\bq}\Vert_\infty}\\&
\text{ for some }\tilde{\bq}=(\bq_0,\bq)\in\Z[1/p]^m\times\Z[1/p]^n \end{aligned}\right\},
\end{equation}
and similarly the Diophantine $\Z$ exponent as 
\begin{equation}\label{def:matrixz}
w(A):=\sup\left\{v ~\biggl|\begin{aligned} &\text{ there exists  infinitely many } \tilde{\bq}=(\bq_0,\bq)\in\Z^{m+n}\\
& \text{ s.t. } \Vert A.\bq+\bq_0\Vert_p\leq \frac{1}{(\Vert\tilde{\bq}\Vert_\infty)^v}\\&
 \end{aligned}\right\}.
\end{equation}
A similar reasoning as before shows that $w_{p}(A)=w(A)+1$.

\section{Connecting Diophantine approximation and homogeneous dynamics}\label{dani}

	We will weaken the hypothesis of Lemma 2.1 of \cite{Kleinbock-exponent} with the same conclusion.
	\begin{lemma}\label{connection}
		Suppose we are given a set $E\subset \R^2$ such that 
		\begin{itemize}
			\item If the second coordinate of $E$ is bounded then the first coordinate cannot converge to $0$ unless it is $0$ ultimately, i.e. $(x_n,z_n)\in E$ such that $|z_n|$ is bounded and $x_n\to 0$ implies that $x_n=0 $ for all but finitely many $n$.
			\item $(0,z)\in E \implies (0,kz)\in E ~ ~ \text{ for infinitely many } k\in\N$. 
		\end{itemize}
		Take $a, b>0$ and $v>\frac{a}{b}$ and define 
		$$c:=\frac{bv-a}{v+1}\Leftrightarrow v=\frac{a+c}{b-c}.$$
		 As before, $p$ is a prime. Then the following are equivalent:
		\begin{enumerate}
			\item There exists $(x,z)\in E$ with arbitrarily large $|z|$ such that $$|x|\leq |z|^{-v}.$$
			\item There exists arbitrarily large $t > 0$ such that for some $(x,z)\in E\setminus\{0\} $ one has 
			$$
			\max(p^{at}|x|,p^{-bt}|z|)\leq p^{-ct}.
			$$
		\end{enumerate}
		
	\end{lemma}
	\begin{proof}
		We first show that $(1)$ implies $(2)$. There exists $(x,z) \in E$ with arbitrary large $|z|$ such that  $|x|\leq |z|^{-v}$. Define $t>0$ such that $p^{-bt}|z|=p^{-ct}$; this is possible   since $b-c>0$.  Then $$
		p^{at}|x|\leq p^{at}|z|^{-v}=p^{at}p^{(b-c)t(-v)}=p^{at}.p^{-(a+c)t}=p^{-ct}$$
		We now show that $(2)$ implies $(1)$. Accordingly, we assume that there exists a sequence of positives $\{t_n\}\to\infty $ such that 
		$$p^{at_n}|x_n|\leq p^{-ct_n} \text{ and } p^{-bt_n}|z_n|\leq p^{-ct_n}.$$
		Therefore $$|x_n|\leq p^{-(a+c)t_n}=p^{-(b-c)vt_n}\leq |z_n|^{-v}.$$
		If $\{z_n\}$ is unbounded then $(1)$ is proved. Suppose then, that $\{z_n\}$ is bounded. Since $\{x_n\}\to 0$, by the hypothesis $x_n = 0$ for all but finitely many $n$. Therefore we have that $(0, z_m)\in E\setminus \{0\}$. By the hypothesis, $(0,kz_m)\in E\setminus\{0\}$ for inifinitely  many $k\in \N$, which will satisfy $(1)$.		 
	\end{proof}
\begin{proposition}\label{connection1}
	For $\by\in\Q_p^n$, the following are equivalent 
	\begin{enumerate}
		\item $\mathbf{y} \in \mathcal{W}^{p}_v$, where $v>{n}$,
		\item there exists arbitrarily large $t>0$ such that $$
		\max\{p^{\frac{nt}{n+1}}|q_0+\bq.\by|_p\Vert\tilde{\bq}\Vert_\infty, p^{-\frac{t}{n+1}}\Vert\bq\Vert_p\Vert\tilde{\bq}\Vert_\infty \}\leq p^{-ct}$$ where $a=\frac{n}{n+1}, b=\frac{1}{n+1}, c=\frac{v-n}{(n+1)(v+1)}\Leftrightarrow v=\frac{n(1+c)+c}{1-(n+1)c} \text{ and } \tilde\bq =(q_0,\bq)\in\Z[1/p]^{n+1}$.
		\end{enumerate}
	\begin{proof}
		We will apply Lemma \ref{connection} to the set 
		$$E= \left\{ |q_0+\bq.\by|_p\Vert\tilde{\bq}\Vert_\infty,\Vert\bq\Vert_p\Vert\tilde{\bq}\Vert_\infty ~\big | ~ \tilde{\bq}=(q_0,\bq)\in\Z[1/p]^{n+1}\right\}.$$ By Lemma \ref{setproperty}, this set satisfies the first hypothesis of Lemma \ref{connection}. Suppose $$(0,\Vert\bq\Vert_p\Vert\tilde{\bq}\Vert_\infty)\in E,$$ then $ |q_0+\bq.\by|_p\Vert\tilde{\bq}\Vert_\infty=0$. For any $u\in\N$ such that $p\nmid u$, 
$$|u.q_0+u\bq.\by|_p\Vert u\tilde{\bq}\Vert_\infty=0$$
and $$\Vert u\bq\Vert_p\Vert u\tilde{\bq}\Vert_\infty=u\Vert\bq\Vert_p\Vert\tilde{\bq}\Vert_\infty$$
implies that $(0,u\Vert\bq\Vert_p\Vert\tilde{\bq}\Vert_\infty)\in E$, giving the second hypothesis of Lemma \ref{connection}. Now this Proposition follows directly from Lemma \ref{connection}.
	\end{proof}
\end{proposition}
\section{Quantitative Nondivergence for flows on homogeneous spaces}\label{margulis}
We begin this section by stating Theorem $2.1$ of \cite{Kleinbock-exponent}. This theorem is an improvement of an original theorem of Kleinbock and Margulis (\cite{KM}). This improvement was the main tool in D. Kleinbock's approach to studying Diophantine exponents of subspaces and their nondegenerate submanifolds. We will use a $p$-adic version of the Theorem, which in turn constitutes an improvement of the nondivergence theorem in \cite{KT}. Nondivergence estimates for flows on homogeneous spaces have a rich history, we refer the reader to \cite{Kleinbock-exponent} and the references therein. 
		\begin{theorem}\label{QND-original} Let 
		${k},N\in
		\Z_+$ and
		$C,\alpha,D  > 0$.
		Suppose that we are given  an 
		$N$-Besicovitch metric space $X$,
		a weighted poset $(\mathcal{B}, \eta)$, a ball $B = B(x,r)$ in
		$X$,  a measure 
		$\mu$ which is $D$-Federer  on $\tilde B= B\big(x,3^mr\big)$, and    a mapping
		$\psi:\mathcal{B}\to C(\tilde B)$, $s\mapsto \psi_s$,   such that the
		following holds:
		
		\begin{enumerate}
			\item $ \ell(\mathcal{B}) \le {k}$;
			\item $\forall\,s\in \mathcal{B}\,,\quad \psi_s$ is $(C,\alpha)$\ on $\tilde B$ with
			respect to
			$\mu$;
			\item $\forall\,s\in \mathcal{B}\,,\quad\|\psi_s\|_{\mu,B} \ge\eta(s)
			$;
			\item 
			$\forall\,y\in \tilde
			B \,\cap\,\supp\mu,\quad\#\{s\in \mathcal{B}\bigm|
			|\psi_s(y)| < \eta(s)\} < \infty$.
		\end{enumerate}
		Then
		$
		\forall\,\varepsilon> 0$ one has
		$$
		\mu\big(B\smallsetminus \Phi(\varepsilon,
		{\mathcal{B}})\big) \le {k}C
		\big(ND^2\big)^{k}
		\varepsilon^\alpha
		\mu(B)\,.
		$$
	\end{theorem} 
	In the following discussion we assume,
	\begin{itemize}
		\item $\mathcal D$ is an integral domain, that is, a  commutative
		ring with
		$1$ and without zero divisors;
		\item $K$ is the quotient field of $\mathcal D$;
		\item  ${\mathcal R}$ is a commutative ring containing
		${K}$ as a subring.
	\end{itemize}	
	The following Theorem is an improvement of Theorem 6.3 of \cite{KT} using the improved quantitative nondivergence i.e. Theorem \ref{QND-original} of D. Kleinbock. We refer the reader to loc.cit. for the definition of norm-like functions. 
	\begin{theorem}\label{QND}
		Let $X$ be a metric space,
		$\mu$ a  uniformly Federer measure  on $X$, and let ${\mathcal D}\subset {K}
		\subset {\mathcal R}$ be as above,
		${\mathcal R}$ being a topological ring. For  $m\in \N$, let  a
		ball $B = B(x_0,r_0)\subset X$ and a continuous map $h:\tilde B \to
		\GL(m,{\mathcal R})$
		be given, where $\tilde B$ stands for $B(x_0,3^mr_0)$. Also let $\nu$ be a
		norm-like function on $\mathcal M({\mathcal R},{\mathcal D},m)$.
		%
		%
		For any $\Delta\in \mathcal P({\mathcal D},m)$ denote by $\psi_\Delta$ the
		function
		$x\mapsto \nu\big(h(x)\Delta\big)$ on  $\tilde B$.
		Now suppose for some
		$C,\alpha > 0$  one has
		
		{\rm(i)} for every $\Delta\in \mathcal P({\mathcal D},m)$, the function $\psi_\Delta$
		is $(C,\alpha)$ on
		$\tilde B$  with respect to
		$\mu$,
		
		{\rm(ii)}  for every $\Delta\in \mathcal  P({\mathcal D},m)$, $\|\psi_\Delta\|_{\mu,B}
		\ge \rho^{\rk\Delta}$,
		
		{\rm(iii)}   $ \forall\,x\in \tilde
		B \,\cap\,\supp\mu,\quad\#\big\{\Delta\in\mathcal P({\mathcal D},m)\bigm|
		\psi_\Delta(x) <
		\rho \big\} < \infty. $
		
		\noindent Then 
		for any  positive $\varepsilon\le \rho$ one has
		$$
		\mu\left(\left\{x\in B \left| \nu \big(h(x)\gamma\big) <  \frac{\varepsilon}{C_{\nu}}
		\text{ for \ }\text{some }\gamma\in {\mathcal D}^m\setminus \{0\}
		\right. \right\}\right)\le mC
		\big(N_{X}D_{\mu}^2\big)^m
		\left(\frac\varepsilon\rho \right)^\alpha
		\mu(B).
		$$
	\end{theorem}
	\begin{proof}
		Take $\eta(\Delta)=\rho^{rk(\Delta)}$. We want to show that 
		$$\Phi\left(\frac{\varepsilon}{\rho},
		{\mathcal{B}}\right) \subset \bigg\{\ x\in B\left| \nu(h(x)\gamma)\geq \frac{\varepsilon}{C_{\nu}}  ,~\forall~ \gamma\ \in\mathcal{D}^m\setminus\{0\}\right.\bigg\}.$$
		Take $x\in \Phi\big(\frac{\varepsilon}{\rho},
		{\mathcal{B}}\big) \cap B$, so there exists a flag $\mathcal F_{x}$.
		Let $$\{0\} =
		\Delta_0 \subsetneq \Delta_1 \subsetneq\dots \subsetneq\Delta_l =
		{\mathcal D}^m$$ be all the elements of $\mathcal F_x \cup
		\big\{\{0\},{\mathcal D}^m\big\}$ such that the following conditions are satisfied:
		\begin{enumerate}
			\item\label{M1} $\frac{\varepsilon}{\rho}\eta(\Delta)\leq |\psi_{\Delta}(x)|\leq \eta(\Delta) \ \ \forall \delta \in \mathcal{F}_x.$\\
			\item\label{M2} $|\psi_{\Delta}(x)|\geq \eta(\Delta)  \ \ \forall \Delta \in \mathcal B(\mathcal F_x).$
		\end{enumerate}

		Pick any $\gamma\in {\mathcal D}^m\setminus\{0\}$. Then there
		exists $i$, $1 \le i \le l$, such that $\gamma\in \Delta_i\setminus
		\Delta_{i-1}$. Now $\gamma\notin \Delta_{i-1}=\mathcal R\Delta_{i-1}\cap \mathcal D^m$ implies that $\gamma\notin \mathcal R\Delta_{i-1}$ since $\Delta_{i-1}$ is primitive, hence
		$g\gamma\notin g{\mathcal R}\Delta_{i-1} = {\mathcal R}g\Delta_{i-1}$ for any
		$g\in\GL(m,{\mathcal R})$.
		Therefore, if one defines
		$\Delta' := {\mathcal D}\Delta_{i-1} + {\mathcal D}\gamma$, in view of  (N2)
		one has
		$$\nu\big(h(x)\Delta'\big) \le C_\nu
		\nu\big(h(x)\Delta_{i-1}\big)\nu\big(h(x)\gamma\big).
		$$
		Further, let $\Delta:= {K}\Delta'\,\cap\, {\mathcal D}^m$. It is a primitive
		submodule containing $\Delta'$ and of rank equal to $\rk(\Delta')$,
		so, by (N1),
		$$\nu\big(h(x)\Delta\big) \le \nu\big(h(x)\Delta'\big).
		$$
		Moreover, it is also contained in $\Delta_i$ and contains $\Delta_{i-1}$, since
		$$
		\Delta_{i-1} = {K}\Delta_{i-1}\,\cap\, {\mathcal D}^m\subset \Delta = {K}\Delta'\,\cap\, {\mathcal D}^m =  {K}\Delta\,\cap\, {\mathcal D}^m \subset
		{K}\Delta_i\,\cap\, {\mathcal D}^m = \Delta_i\,.
		$$
		Therefore it is comparable to any element of
		$\mathcal F_x$, i.e.~belongs to $\mathcal F_x \cup \mathcal P(\mathcal F_x)$.
		Then one can use
		properties (\ref{M1}) and (\ref{M2}) above to deduce
		that
		\begin{align*}
		\nu(h(x)\Delta))=|\psi_{\Delta}(x)| = \nu\big(h(x)\Delta\big) &\ge \min\left(\frac{\varepsilon}{\rho}\rho^{\rk(\Delta)},
		\rho^{\rk(\Delta)}\right) \\ &=
		\frac{\varepsilon}{\rho}\rho^{\rk(\Delta)}=\frac{\varepsilon}{\rho}\rho^{\rk(\Delta_{i-1})+1} 
		=\varepsilon\rho^{\rk(\Delta_{i-1})}\,,
		\end{align*}
		and then, in view of (6.5) and (6.6), conclude that
		\begin{align*}
		\nu\big(h(x)\gamma\big) \ge
		{\nu\big(h(x)\Delta^\prime\big)}/{C_\nu \nu\big(h(x)\Delta_{i-1}\big)}&\ge{\nu\big(h(x)\Delta\big)}/{C_\nu \nu\big(h(x)\Delta_{i-1}\big)}  \\&\ge
		\varepsilon\rho^{\rk(\Delta_{i-1})}/C_\nu\rho^{\rk(\Delta_{i-1})}  \\&\ge
		\varepsilon/C_{\nu}.
		\end{align*}

	\end{proof}

	We recall Lemma 8.1 from \cite{KT}. It is proved for an arbitrary, finite set of places $S$ of $\Q$. We only need it for the case $S = \{\infty, p\}$ for a prime $p$. 
	
	\begin{lemma} The function  $\nu:\mathcal M(\Q_S,{\Z}_S,m)\to\R_+$
		given by $\,\nu(\Delta) = cov(\Delta)$, with $\,cov(\cdot)$ as
		defined earlier
		is norm-like, with $C_\nu = 1$.
	\end{lemma}
	As a consequnece of Theorem \ref{QND} we get the following refined version of Theorem 8.3 of \cite{KT}
	\begin{theorem}\label{QND2}
		Let $X$ be a 
		be a Besicovitch metric space,
		$\mu$ a  uniformly Federer measure  on $X$, and let
		$S$
		be as above.  For  $m\in \N$, let  a
		ball $B = B(x_0,r_0)\subset X$ and a continuous map $h:\tilde B
		\to \GL(m,\Q_S)$ be given, where $\tilde B$ stands for
		$B(x_0,3^mr_0)$.
		Now suppose that for some
		$C,\alpha > 0$ and $0 < \rho  < 1$  one has
		
		{\label{rm1}\rm(i)} for every $\,\Delta\in \mathcal P({\Z}_S,m)$, the function
		$\cov\big(h(\cdot)\Delta\big)$ is $(C,\alpha)$ good \ on $\tilde B$  with
		respect to $\mu$;
		
		{\label{rm2}\rm(ii)}  for every $\,\Delta\in \mathcal P({\Z}_S,m)$,
		$\sup_{x\in B\cap \supp\mu}\cov\big(h(x)\Delta\big) \ge \rho^{rk(\Delta)}$.
		
		Then 
		for any  positive $ \varepsilon\le
		\rho$ one has
		$$
		\mu\left(\big\{x\in B\bigm| \delta \big(h(x){\Z}_S^m\big) < \varepsilon
		\big\}\right)\le mC \big(N_{X}D_{\mu}^2\big)^m
		\left(\frac\varepsilon \rho \right)^\alpha \mu(B)\
		$$
	\end{theorem}
	\begin{proof}
		The proof goes line by line as Theorem (8.3)\cite{KT} using Theorem \ref{QND} in place of theorem (6.3) of \cite{KT}.
	\end{proof}
	
\section{Applying nondivergence estimates}\label{kleinbock}
For any $\by\in\Q_p^{n+1}$ we associate a lattice $u_{\by}\mathcal{D}^{n+1}$ in $(\Q_p\times \R)^{n+1}$, where $u_{\by}$ is defined as
$$u_y^p=
\begin{bmatrix}
1 & y_1 &\cdots &y_n\\
\ &  &I_n
\end{bmatrix}$$ and $u_{y}^\infty=I_{n+1}$.
For $t\in\N$ define
$$g_t^p=\begin{bmatrix}
p^{-t} &0\\
0 &I_n
\end{bmatrix}
\text{ and }
   g_t^\infty = \diag(p^{-\frac{t}{n+1}},\cdots, p^{-\frac{t}{n+1}}).$$
So $g_tu_y\in \GL_{n+1}(\Q_p\times \R)$.\\

\noindent
We now proceed to connect dynamics with Diophantine properties. 
For any $v>n$ and  $\frac{1}{n+1}>d>c=\frac{v-n}{(n+1)(v+1)}$, define $v_d=\frac{n(1+d)+d}{1-(n+1)d}$. Hence $\frac{n(1+c)+c}{1-(n+1)c}=v<v_d$.
By Proposition \ref{connection1},
$$
\mathcal W^p_{v_d}=\left\{\by\in \Q_p^n \bigg| 
\begin{aligned}
\exists \text{ arbitrarily large }  t>0 \text{ such that }\\
\max \{p^{\frac{nt}{n+1}}|q_0+\bq.\by|_p\Vert\tilde{\bq}\Vert_\infty, p^{-\frac{t}{n+1}}\Vert\bq\Vert_p\Vert\tilde{\bq}\Vert_\infty \}\leq p^{-dt} \\
\text{ for some }
\tilde{\bq}\in\mathcal{D}^{n+1}
\end{aligned}
\right\}.
$$
Now define the set 
$$
\widetilde{\mathcal W_{v_d}}:=\Bigg\{\by\in \Q_p^n \bigg| 
\begin{aligned}
~\exists \text{ arbitrarily large }  t\in \N \\\text{ such that }
\delta(g_tu_{\by}\mathcal{D}^{n+1})\leq p^{-dt}
\end{aligned}
\Bigg\}.
$$
Then we have 
\begin{lemma}\label{d_tilde_d}
	Let $\nu$ be a measure on $\Q_p^n$. Then the following statements are equivalent: 
	
	\begin{enumerate}
		\item $\nu(\mathcal W^p_{v_d})=0 ~~ \forall~d > c.$\\
		\item $\nu (\widetilde{\mathcal W_{v_d}})=0 ~~\forall~d > c.$
	\end{enumerate}
\end{lemma}
\begin{proof}
	We first prove that $(1)$ implies $(2)$. Let $\by\in\widetilde{\mathcal W_{v_d}},$ then there exists arbitrarily large $t\in\N$ such 
	that \begin{equation}\label{1}
	\delta(g_tu_{\by}\mathcal{D}^{n+1})\leq p^{-dt}.
	\end{equation}
	Observe that $$(g_t u_{\by}\mathcal{D}^{n+1})_\infty=\{(p^{-\frac{t}{n+1}}q_0,\cdots,p^{-\frac{t}{n+1}}q_n) \left| q_i\in\mathcal D \right.\},$$
	$$(g_t u_{\by}\mathcal{D}^{n+1})_p=\{(p^{-t}(q_0+\bq.\by),q_1,\cdots,q_n)\left| q_i\in\mathcal{D}\right.\}$$
	and
$$	\delta(g_t u_{\f(x)}\mathcal{D}^{n+1})=\min_{\tilde\bq\in\mathcal{D}^{n+1}\setminus\{0\}}c(g_tu_{\f(x)}\tilde{\bq}).$$
	So from (\ref{1}) there exist infinitely many $t\in\N$ such that$$
	\max\{p^t\vert \bq.\by+q_0\vert_p,\Vert\bq\Vert_p\}.p^{-\frac{t}{n+1}}\Vert\tilde{\bq}\Vert_\infty\leq p^{-dt}$$
	which implies that
	$$\max\{p^{\frac{nt}{n+1}}\vert \bq.\by+q_0\vert_p \Vert\tilde{\bq}\Vert_\infty,p^{-\frac{t}{n+1}}\Vert\bq\Vert_p.\Vert\tilde{\bq}\Vert_\infty\}\leq p^{-dt}  
	$$
	for some $\tilde{\bq}\in\mathcal{D}^{n+1}$. 
	 Hence $\widetilde{\mathcal W_{v_d}}\subset \mathcal W^p_{v_d} ~~\forall ~d>c.$\\
	We now prove $(2)$ implies ($1$). We claim that $\mathcal W^p_{v_d}\subset \widetilde{\mathcal W_{v_ d^\prime}} $ for some $d> d^\prime>c$. Note that
	$$\begin{aligned}\Vert\bq\Vert_p\Vert\tilde{\bq}\Vert_\infty&\leq p^{-(d-\frac{1}{n+1})t}\\
	&\leq p^{-(d-\frac{1}{n+1})}p^{-(d-\frac{1}{n+1})[t]}\\
	&\leq \frac{p^{-(d-\frac{1}{n+1})}}{p^{(d-d^\prime)[t]}}p^{-(d^\prime-\frac{1}{n+1})[t]} \text{ for some  }  d> d^\prime>c\\
	& \leq p^{-(d^\prime-\frac{1}{n+1})[t]} \text{ for large enough } t>0.
	\end{aligned}
	$$
	Finally, 
	$$p^{[t]}|\bq.\by+q_0|_p\Vert\tilde{\bq}\Vert_\infty\leq p^t|\bq.\by+q_0|_p\Vert\tilde{\bq}\Vert_\infty\leq p^{-dt}\leq p^{-d[t]}\leq p^{-d^\prime[t]}$$ 
	for some $\tilde{\bq}\in\mathcal{D}^{n+1}$. 
	Therefore we have that $\mathcal W^p_{v_d}\subset \widetilde{\mathcal W_{v_ d^\prime}} $ and the conclusion follows.
\end{proof}

\begin{proposition}\label{prop1}
Take $\mathcal{R}=\Q_p\times \R$ and $\mathcal{D}=\Z[1/p]$and $v>n, c:=\frac{v-n}{(n+1)(v+1)} $ as defined in Proposition \ref{connection1}. Let $X$ be a Besicovitch metric space and $\mu$ be a uniformly Federer measure on $X$. Denote $\tilde B:=B(x,3^{n+1}r)$. Suppose we are given a continuous function $\f : X\mapsto \Q_{p}^{n}$ and $C,\alpha >0$ with the following properties\\
		{\label{b1}\rm{(i)}}~$x\mapsto \cov(g_t u_{\f(x)}\Delta) $is $(C,\alpha)$ good  with respect to  $\mu\text{ in }\tilde B  ~\forall ~\Delta \in \mathcal P(\mathcal D,n+1),$\\
		{\label{b3}\rm{(ii)}}  for any $d>c $ there exists $T=T(d) >0$ such that for any integer  $t\geq T$ and any $\Delta \in \mathcal P(\mathcal D,n+1)$ one has 
		\begin{equation}\label{A2}
		\sup_{x\in B\cap \supp\mu}\cov (g_t u_{\f(x)}\Delta)\geq p^{-(\rk\Delta)dt}.
		\end{equation}
		Then $w_p(f_*\mu|_B)\leq v$.

\end{proposition}
\begin{proof}
		We will check that the map $h=g_tu_{\f}$ satisfies the assumptions of Theorem \ref{QND2} with respect to the measure $\nu=\f_*\mu|_B$ where 
	(\ref{b1})  is the same as (\rm{i}) of Theorem \ref{QND2} with $m=n+1$. To check the second condition 
	take $\rho =p^{-\frac{c+d}{2}t} \text{ for } d>c$. Then (\ref{b3}) gives (\rm{ii}) of Theorem \ref{QND2} for all integer $t>T(\frac{c+d}{2}).$ Therefore by Theorem \ref{QND2} 
	\begin{align}\label{borel_cantelli}
	&\mu\left(\bigg\{x\in B\left | \delta (g_t u_{\f(x)}\mathcal D^{n+1})<p^{-dt}\right.\bigg\}\right)\\
	&\leq (n+1)C(N_X D_{\mu}^2)^{n+1}(p^{-dt}p^{\frac{c+d}{2}t})^\alpha \mu(B)\\
	&= const.p^{-\alpha\frac{d-c}{2}t}
	\end{align} for all but finitely many $t\in \N$.\\
	Using the Borel-Cantelli lemma and (\ref{borel_cantelli}) we have $\nu(\widetilde{\mathcal W_{v_d}})=0 ~\forall~d>c$, which  again gives $\nu({\mathcal W^p_{v_d}})=0 ~ \forall~d>c$ by Lemma \ref{d_tilde_d}. Since $\lim_{k \to \infty}d_k = c$, we have that  $\lim_{k \to \infty}v_{d_k} = v$ and now using the formula for $v_d$, we have that  for any $u>v$ there exists $\frac{1}{n+1}>d>c$ such that $u>v_{d_k}$. So $\nu(\mathcal W^p_u)=0$ because $\mathcal{W}^p_u\subset \mathcal{W}^p_{v_{d_{k}}}$. Hence $w_p(\nu)=w_p(f_*\mu|_B)\leq v$.
\end{proof}
Let us now show that condition $(2)$ in Proposition \ref{prop1} is in fact necessary to conclude Proposition \ref{prop1}. This a $p$-adic version of Lemma (4.1) of \cite{Kleinbock-exponent}.
\begin{lemma}\label{necessary}
	Let $\mu$ be a measure on $B$ and take $c,v>0$ as before. Let $\f: B\mapsto \Q_p^n$ be such that condition (\rmfamily{ii}) in Proposition \ref{prop1} does not hold. Then 
	\begin{equation}
	\f(B\cap\supp\mu)\subset \mathcal{W}_u^{p} \text{ for some } u>v.
	\end{equation}
\end{lemma}
\begin{proof}
	There exists $d>c$ such that condition(2) in Proposition \ref{prop1} does not hold. So there exists a $1\leq j\leq (n+1)$ such that there exists a sequence of natural numbers $t_i\to \infty$ and corresponding $\Delta_i$ of rank $j$ such that 
	$$
	\sup_{x\in B\cap \supp\mu}\cov (g_{t_i} u_{\f(x)}\Delta_i)\leq p^{-(\rk\Delta_i)dt_i}=p^{-jdt_i}.
	$$
	Now consider the ball \begin{align*}
	D&=D_\infty\times D_p\\
	&=\left\{ x^\infty\in\R^{n+1}\left| ||x^\infty||_\infty\leq p^{-dt_i} \right.\right\}\times
	\left\{ x^{(p)}\in\Q_p^{n+1}\left| \begin{aligned}
	&|x^{(p)}_1|_p\leq 1\\
	&|x^{(p)}_2|_p\leq 1\\
	&\vdots \\
	&|x^{(p)}_{n+1}|_p\leq 1
	\end{aligned}\right.\
	\right	\}.
	\end{align*}
	Denote $\Delta=g_{t_i}u_{\f(x)}\Delta_i $ and $ D\cap \Q_S\Delta=D_1$, a ball in $\Q_S\Delta$. Then 
	$\lambda_S(D_1)=\mu_S(\pi\inv(D_1))$ where $\mu_S$ is the Haar measure on $\Q_S^j$ and $$\pi:\Q_S^j\mapsto\Q_S\Delta=\Q_Sv_1+\cdots+\Q_Sv_j$$ where the $v_1,v_2,\cdots,v_j$ are taken such that $v_1^{\infty},\cdots,v_j^{\infty}$ form a orthonormal basis of $(\Q_s\Delta)_\infty$ and $$(\Q_S\Delta)_p\cap \Z_p^{n+1}=\Z_pv_1^{(p)}+\cdots+\Z_pv_j^{(p)}.$$ Note that $\lambda_S$ is the normalized Haar measure on $\Q_S\Delta$. Now consider
	\begin{align*}
	\pi\inv D_1=&\left\{ x\in \Q_s^j\left|\begin{aligned}
	&\Vert x_1^\infty v_1^\infty+\cdots+x_j^\infty v_j^\infty\Vert_\infty\leq p^{-dt_i}\\
	&|(x_1^{(p)}v_1^{(p)}+\cdots+x_j^{(p)}v_j^{(p)})_k|_p\leq 1 \ \forall k=1,\cdots,n+1
	\end{aligned}\right.  \right\},\\
	&=\tilde D_1^\infty\times\tilde D_1^{(p)}
	\text{ where, }\\
	\tilde D_1^\infty=&\left\{x^\infty \in \R^j\left|||x_1^\infty v_1^\infty+\cdots+x_j^\infty v_j^\infty||_\infty\leq p^{-dt_i}\right. \right\}\\
	\tilde D_1^{(p)}=&\left\{x^{(p)}\in\Q_p^j\left|
	\begin{aligned}
	&|(x_1^{(p)}v_1^{(p)}+\cdots+x_j^{(p)}v_j^{(p)})_k|_p\leq 1 \ \forall k=1,\cdots,n+1
	\end{aligned}
	\right.\right \}  .\end{align*}
	Since $v_1^\infty,\cdots,v_j^\infty $ are orthonormal, we have that 
	$\mu_\infty(\tilde D_1^\infty)=2^j.p^{-jdt_i}$.\\
	
	 \noindent On the other hand, since $B(0, 1)\times B(0,1)\times\cdots B(0,1)\subset\tilde D_1^{(p)}$,  $$1\leq\mu_p(\tilde D_1^{(p)}).$$
	So $$\mu(\pi\inv D_1)=\mu_\infty(\tilde D_1^\infty)\times \mu_p(\tilde D_1^{(p)})\geq 2^jp^{-jdt_i},$$ 
	and this gives  
	$$\lambda_S(D_1)\geq 2^j p^{-jdt_i}\geq 2^j\cov(g_{t_i}u_{\f(x)}\Delta_i)=2^j\cov(\Delta) \ \forall x\in B\cap \supp\mu.$$ 
	Then by the $S$-arithmetic version of Minkowski's theorem (\cite{KT}, Lemma 7.7) there is a nonzero vector $v\in g_{t_i}u_{\f(x)}\Delta_i \cap D_1$ i.e. there exists $\tilde \bq=(q_0,q_1,\cdots,q_n)\in(\Z[1/p])^{n+1} $ such that 
	\begin{equation}
	p^{-\frac{t_i}{n+1}}\Vert\tilde\bq\Vert_\infty \leq p^{-dt_i},
	\end{equation}
	and
	\begin{equation}\label{p-equations}
	\begin{aligned}
	&p^{t_i}|q_0+q_1f_1(x)+\cdots+f_n(x)|_p\leq 1\\
	&|q_1|_p\leq 1\\
	&\vdots\\
	&|q_n|_p\leq 1.
	\end{aligned}
	\end{equation}
 Therefore $$\begin{aligned}
	&p^{-\frac{t_i}{n+1}}\Vert\tilde \bq\Vert_\infty\Vert\bq\Vert_p\leq p^{-\frac{t_i}{n+1}}\Vert\tilde \bq\Vert_\infty \leq p^{-dt_i}
	\text{ and }\\
	&p^{\frac{nt_i}{n+1}}\vert q_0+\bq.\f(x)\vert_p\Vert\tilde{\bq}\Vert_\infty\leq p^{-dt_i}
	\end{aligned}$$ holds for infinitly many $t_i\in\N$.
	Hence by equivalence (\ref{connection1})  $$
	\forall \  x\in\supp \mu \cap B \text{ such that } \f(x)\in\mathcal W^{p}_{v_{d}} \text{ for some } d >c.$$ Thus there exists $v_{d}>v$ such that $$
	\f(B\cap\supp\mu)\subset \mathcal{W}^{p}_{v_d}.$$
\end{proof}

Throughout the rest of the paper we are going to denote $\mathcal{R}=\Q_p\times\R$ and $\mathcal{D}=\Z[1/p]$. One can associate to any nonzero submodule $\Delta\subset \mathcal{D}^{n+1}$ of rank $j$,  an element $\bw$ of $\bigwedge^j(\mathcal{D}^{n+1})$ such that $\cov(\Delta)=c(\bw)$ and $\cov(g_tu_{\by}\Delta)=c(g_tu_{\by}\bw)$.  We take
$\be_0,\be_1,\cdots,\be_n\in\mathcal{R}^{n+1}$ as the standard basis where $\be_i=(\be_i^p,\be_i^\infty)$ and $\{\be_i^p\} ,\{\be_i^\infty\}$ are the standard basis of $\Q_p^n$, $\R^n$ respectively. We will use the standard basis $\{\be_I=\be_{i_1}\wedge\cdots\wedge \be_{i_j}~|I\subset\{0,\cdots,n\}\text{ and } i_1<i_2<\cdots<i_j \}$ of $\bigwedge^j\mathcal{R}^{n+1}$. Thus we can write any $\bw\in\mathcal{D}^{n+1}$ as $\bw=\sum w_I\be_I$, where $w_I\in\mathcal{D}$.
Let us note the action of unipotent flows on the coordinates of $\bw$. 
\begin{enumerate}
	\item $u_{\by}^p$ leaves $\be_0^p$ invariant and sends $\be_i^p$ to $y_i\be_0^p+\be_i^p$ for $i\geq 1$.\\
	\item $u_{\by}^\infty =I_{n+1}$ leaves everything invariant.
\end{enumerate}
Therefore $$
u_\by^p(\be_I^p)=\left\{\begin{aligned}
&\be_I^p  &\text{ if } 0\in I\\
&\be_I^p+\sum_{i\in I} \pm y_i\be_{I\setminus\{i\}\cap\{0\}} & \text{ if } 0\notin I.
\end{aligned}
\right.$$
Observe that under the action of $g_t^p$, $\be_i^p$s' are invariant for $i\geq 1$ and $\be_0^p$ is an eigenvector with eigenvalue $p^{-t}$. Therefore
$$
g_t^pu_\by^p(\be_I^p)=\left\{\begin{aligned}
&p^{-t}\be_I^p  &\text{ if } 0\in I\\
&\be_I^p+p^{-t}\sum_{i\in I} \pm y_i\be_{I\setminus\{i\}\cap\{0\}} & \text{ if } 0\notin I.
\end{aligned}
\right.$$
On the otherhand $u_{\by}^\infty = \Id$ and each $\be_i^\infty$ is an eigenvector with eigenvalue $p^{-\frac{t}{n+1}}$. Thus 
$$
g_t^\infty u_{\by}^\infty\be_I^\infty=p^{-\frac{tj}{n+1}}\be_I^\infty.
$$
Therefore for $\bw\in\bigwedge^j(\mathcal{D}^{n+1}), \bw=\sum w_I\be_i$ where $w_I\in\mathcal{D}$ we get 
\begin{equation}
(g_tu_\by\bw)^p=\sum_{0\notin I}w_I\be_I^p+p^{-t}\sum_{0\in I}\big(w_I+(\sum_{i\notin I}\pm w_{I\setminus\{0\}\cup\{i\}}y_i)\big)\be_I^p,
\end{equation}
 and \begin{equation}
 (g_tu_{\by}\bw)^\infty=p^{-\frac{tj}{n+1}}\sum w_I\be_I^\infty.
 \end{equation}

It will be convenient for us to use the following notations, $$\bc(\bw)=\begin{pmatrix}
\bc(\bw)_0\\
\bc(\bw)_1\\
\vdots\\
\bc(\bw)_n
\end{pmatrix},$$
where $\bc(\bw)_i=\sum\limits_{\substack{J\subset\{1,\cdots,n\} \\\# J=j-1}} w_{J\cup\{i\}}\be^p_J\in \bigwedge^{j-1}(V_0)$ and $V_0$ is the subspace of $\Q_p^{n+1}$ generated by $\be^p_1,\cdots,\be^p_n$. 
We may therefore write 
\begin{equation}
(g_tu_\by\bw)^p=\sum_{0\notin I}w_I\be_I^p+p^{-t}\big(\be_0^p\wedge\sum_{i=0}^n y_i\bc(\bw)_i\big)\\
=\pi(\bw)+p^{-t}\be_0^p\wedge\tilde{\by}\bc(\bw),
\end{equation}
where $y_0=1$ and $\pi$ is the orthogonal projection from $\bigwedge ^j(\Q_p^{n+1})\mapsto \bigwedge^j(V_0)$.
If $\bw$ corresponds to $\Delta\subset\mathcal{D}^{n+1}$, a submodule of rank $j$ then we have that $$\begin{aligned}
\cov(g_tu_{\by}\Delta) &=c(g_t{u_\by}\bw)\\&=
\max\bigg( p^t\Vert\sum_{i=0}^n y_i \bc(\bw)_i\Vert_p, \Vert \pi(\bw)\Vert_p,\bigg).p^{-\frac{tj}{n+1}}\Vert \bw\Vert_\infty\\
&=\max\bigg( p^{t-\frac{tj}{n+1}}\Vert \sum y_i\bc(\bw)_i \Vert_p\Vert\bw\Vert_\infty, p^{-\frac{tj}{n+1}}\Vert \pi(\bw)\Vert_p\Vert \bw\Vert_\infty\bigg).
\end{aligned}
$$
Thus,
\begin{equation}\label{covolume}
	\begin{aligned}
	&\sup_{x\in B\cap \supp\mu} \cov (g_tu_{\f(x)}\Delta)\\
	&=\max\bigg( p^{t-\frac{tj}{n+1}}\sup_{x\in B\cap \supp\mu}\Vert \tilde\f(x)\bc(\bw) \Vert_{p}\Vert\bw\Vert_\infty, p^{-\frac{tj}{n+1}}\Vert \pi(\bw)\Vert_{p}\Vert \bw\Vert_\infty\bigg),
	\end{aligned}
\end{equation}
where $\tilde{\f}=(1,f_1,\cdots,f_n)$. Note that condition (\ref{A2}) can be written as 
\begin{equation}
\forall~d > c ~ ~\exists ~T>0 \text{ such that } \forall~t\geq T(d), \forall~ j=1, \cdots, n 
\text{ and } \forall ~ \bw.
\end{equation}
Now suppose that the $\Q_p$-span of the restrictions of $1, f_1, \cdots, f_n$ to $B\cap \supp\mu$ has dimension $s+1$ and choose $g_1, \cdots, g_s :B\cap\supp \mu\mapsto \Q_p$ such that $1, g_1, \cdots, g_s$ form a basis of the space.  Therefore there exists a matrix $R=(r_{i,j})_{(s+1)\times(n+1)} $ such that $\tilde{\f}(x)=\tilde{\bg}(x)R ~ \forall~ x\in B\cap \supp \mu$ where $\tilde{\bg}=(1, g_1,\cdots, g_s)$. We can rewrite $$\sup_{x\in B\cap \supp\mu}\Vert \tilde\f(x)\bc(\bw) \Vert_{p}=\sup_{x\in B\cap \supp\mu}\Vert \tilde\bg(x)R\bc(\bw) \Vert_{p}.$$  
Thus we have that (\ref{A2}) is equivalent to
\begin{equation}\label{covolume1}\begin{aligned}
& \text{ for any } d>c ~
\exists ~T=T(d) >0 \text{ such that for any integer  }t\geq T \\& \forall~j=1,\cdots, n \text{ and } \forall~\bw\in \bigwedge^{j}\mathcal{D}^{n+1},\text{ one has } \\&
\max\bigg( p^{t-\frac{tj}{n+1}}\Vert R\bc(\bw) \Vert_{p}\Vert\bw\Vert_\infty, p^{-\frac{tj}{n+1}}\Vert \pi(\bw)\Vert_{p}\Vert \bw\Vert_\infty\bigg)\geq p^{-jdt},
\end{aligned}
\end{equation} by equivalence of norms since $1, g_1, \cdots, g_s$ are linearly independent.
In order to get rid of the auxiliary variable $t$ we want to apply lemma \ref{connection}. Consider the components of $R\bc(\bw)$:
$$\begin{aligned}
(R\bc(\bw))_i &=\sum_{k=0}^nr_{ik}\bc(\bw)_i\\
&=\sum_{k=0}^nr_{ik}\sum\limits_{\substack{J\subset\{1,\cdots,n\} \\\# J=j-1}} w_{J\cup\{k\}}e_J^p\\
&=\sum\limits_{\substack{J\subset\{1,\cdots,n\} \\\# J=j-1}}\big(\sum_{k=0}^n r_{ik} w_{J\cup\{k\}}\big)e_J^p.\end{aligned}
$$
So we have that $$\Vert R\bc(\bw)\Vert_p=\max_{i=0, \cdots, s}\max\limits_{\substack{J\subset\{1,\cdots,n\} \\\# J=j-1}}\vert\sum_{k=0}^n r_{ik}w_{J\cup\{k\}}\vert_p.
$$
Now consider the set $$\begin{aligned}
\mathcal{E}&:=\big\{(\Vert R\bc(\bw)\Vert_p\Vert\bw\Vert_\infty, \Vert\pi(\bw)\Vert_p\Vert\bw\Vert_\infty)~\vert~ \bw\in\mathcal{S}_{n+1, j}\big\}\\&=\big\{(\max_{i=0, \cdots, s}\max\limits_{\substack{J\subset\{1,\cdots,n\} \\\# J=j-1}}\vert\sum_{k=0}^n r_{ik}w_{J\cup\{k\}}\vert_p\Vert\bw\Vert_\infty, \Vert\pi(\bw)\Vert_p\Vert\bw\Vert_\infty)~|~\bw\in\mathcal{S}_{n+1, j}\big\}.\end{aligned}
$$
\begin{lemma}
	The set $\mathcal{E} $ as above satisfies the hypotheses of Lemma \ref{connection}.
\end{lemma}
\begin{proof}
	Let $\bw_m\in \mathcal{S}_{n+1, j}$ be a sequence such that $$\Vert\pi(\bw_m)\Vert_p\Vert\bw_m\Vert_\infty\leq M~ \forall~ m$$ for some $M>0$ and $$\max_{i=0, \cdots, s}\max\limits_{\substack{J\subset\{1,\cdots,n\} \\\# J=j-1}}\vert\sum_{k=0}^n r_{ik}w^{(m)}_{J\cup\{k\}}\vert_p\Vert\bw_m\Vert_\infty\to 0$$ as $m\to\infty$ where $\bw_m=\sum w^{(m)}_Ie_I$. 
	\noindent
	For $J\subset \{1, \cdots, n\},$ such that $ \# J=j-1$, denote $\tilde\bw_J^{(m)}:=(w^{(m)}_{J\cup\{0\}}, \cdots, w^{(m)}_{J\cup\{n\}}) $.
	 We have that for every $m, J $ and $k=1, \cdots, n,$
	$$\begin{aligned}&\vert w^{(m)}_{J\cup\{k\}}\vert_p\Vert \tilde{\bw}^{(m)}_J\Vert_\infty\leq\vert w^{(m)}_{J\cup\{k\}}\vert_p\Vert \bw_m\Vert_\infty\leq M.
	\end{aligned}$$
	 
	Moreover for $\varepsilon>0 ~\exists~N_{\varepsilon}\in\N$ such that $\forall~m\geq N_{\varepsilon} $, $$\begin{aligned}
	& \vert r_{00} w^{m}_{J\cup\{0\}}\vert_p\Vert\bw_m\Vert_\infty\\
	&\leq \max\big(\varepsilon,\big \vert\sum_{k=1}^n r_{im}w^{(m)}_{J\cup\{k\}}\big\vert_p\Vert\bw_m\Vert_\infty\big)\\
	&\leq \max\big(\varepsilon, \Vert R\Vert_p\Vert\pi(\bw_m)\Vert_p\Vert\bw_m\Vert_\infty\big)\\&
	\leq\max(\varepsilon, \Vert R\Vert_p M).
	\end{aligned} $$ 
	Since $\tilde\f(x) =\tilde{\bg}(x) R ~~ \forall~x\in B\cap\supp\mu$ we have $r_{00}=1 \text{ and } r_{i0}=0$ otherwise.
	This implies that 
	$$\vert w^{(m)}_{J\cup\{0\}}\vert_p\Vert \tilde{\bw}^{(m)}_J\Vert_\infty\leq\vert w^{m}_{J\cup\{0\}}\vert_p\Vert\bw_m\Vert_\infty\leq M_1$$ 
	for some $M_1>0$ and for every $J$. Therefore, for 
	$\tilde{\bw}_J^{(m)}$, 
	$$\Vert \tilde{\bw}_J^{(m)}\Vert_p \Vert\tilde{\bw}_J^{(m)}\Vert_\infty\leq M_1,$$ 
	for some $M_1>0$ and 
	$$\vert\sum_{k=0}^n r_{ik}w^{(m)}_{J\cup\{k\}}\vert_p\Vert\tilde{\bw}^{(m)_J}\Vert_\infty\to 0.$$ 
	Now applying Lemma \ref{setproperty} we can conclude that first hypothesis is satisfied. The second hypothesis is satisfied since $(0, \Vert R\bc(\bw)\Vert_p\Vert\bw\Vert_\infty)\in\mathcal{E}$ implies that for $u\in\N\text{ with } p\nmid u, (0, u\Vert R\bc(\bw)\Vert_p\Vert\bw\Vert_\infty)\in\mathcal{E}$.
\end{proof}
Here, in terms of Lemma \ref{connection}, $a= \frac{n+1-j}{n+1}, b=\frac{j}{n+1}$. Since the set $\mathcal{E}$ satisfies the hypotheses of Lemma \ref{connection}, we conclude that the condition (\ref{covolume1}) is equivalent to
\begin{equation}\label{covolume2}
\begin{aligned}
&\forall ~j=1,\cdots,n \text{ and }\forall~\bw\in\mathcal{S}_{n+1,j}, \forall~d>c=\frac{v-n}{(n+1)(v+1)}\text{ where } v\geq n \\& \exists~ u_d=\frac{\frac{n+1-j}{n+1}+dj}{\frac{j}{n+1}-dj}>\frac{\frac{n+1-j}{n+1}+cj}{\frac{j}{n+1}-cj}=\frac{v-j+1}{j} \text{ such that for arbitrarily large }\\& \Vert\pi(\bw)\Vert_p\Vert\bw\Vert_\infty,  \text{ we have } \Vert R\bc(\bw)\Vert_p\Vert\bw\Vert_\infty>(\Vert\pi(\bw)\Vert_p\Vert\bw\Vert_\infty)^{-u_d}.
\end{aligned}
\end{equation}
Moreover $\lim_{k \to \infty}d_k  = c$ implies that  $\lim_{k \to \infty} u_{d_k} = \frac{v-j+1}{j}$. Therefore condition (\ref{covolume2}) is equivalent to \begin{equation}\label{covolume3}
\begin{aligned}
&\forall~ j=1,\cdots, n~, \forall u>\frac{v-j+1}{j} ~\text{ and } \forall \bw\in\mathcal{S}_{n+1,j}
,  \text{ for all arbitrary} \\&\text{ large }\Vert\pi(\bw)\Vert_p\Vert\bw\Vert_\infty  \text{ we have } \Vert R\bc(\bw)\Vert_p\Vert\bw\Vert_\infty>(\Vert\pi(\bw)\Vert_p\Vert\bw\Vert_\infty)^{-u}.
\end{aligned}
\end{equation}
The proof of the Theorem \ref{main theorem} now goes as in \cite{Kleinbock-exponent}, we repeat it for the sake of completeness. Note that we have that $R$ is a matrix depending on the ball $B$, the measure $\mu$ and the map $\f$ such that 
$$(\ref{covolume3}) \text{ holds } \iff \text{ so does } (\ref{covolume1}) \iff \text{ so does } (\ref{A2}).$$

 Suppose \begin{equation}{\label{subspace}}\mathcal{L}=<\f(B\cap \supp\mu)>_a.\end{equation} Let $\dim\mathcal{L}=s$ and let
\begin{equation}{\label{R-type}}\mathbf{h} : \Q_p^s\to \mathcal{L} \text{ be an affine isomorphism, and }\tilde{\mathbf{h}}(\bx)= \tilde{\bx}R, \bx\in\Q_p^s,
\end{equation} 
where as usual we have that $\tilde{\mathbf{h}}:=(1,h_1, \cdots, h_n) $ and $\tilde{\bx}:=(1, x_1, \cdots, x_s)$. Then $\bg= \mathbf{h}\inv\circ \f$ generates the space of $\Q_p$ span of the restrictions of $1, f_1, \cdots, f_n$ to $B\cap \supp\mu$ and satisfies $\tilde{\f}(x)=\tilde{\mathbf{g}}(x)R ~\forall ~x\in B\cap\supp\mu$. Therefore condition (\ref{covolume3})$\iff$ (\ref{covolume1}) $\iff  $condition (\ref{A2}) becomes a property of the subspace and in particular $R$ can be chosen uniformly for all measures $\mu$, ball B and measure $\mu$ and $\f$ the function as long as (\ref{subspace}) holds.
\begin{theorem}\label{main theorem}
	 Let $\mu$ be a Federer measure on a Besicovitch metric space $X, \mathcal{L}$ an affine subspace of $\Q_p^n$, and let $\f : X\to \mathcal{L}$ be a continuous map such that $(\f, \mu)$ is good and nonplanar i.e (\ref{subspace})  holds for all open balls B with $\mu(B)>0$. Then the following are equivalent for $v\geq n$ and $c_v=\frac{v-n}{(n+1)(v+1)},$\\
	 \rm{(i)} \begin{equation}\label{c1}
	\big\{x~\in\supp\mu ~| \f(x)\notin\mathcal{W}_u^{p}\big\} \text{ is nonempty for any } u>v.
	 \end{equation}
	 \rm{(ii)}\begin{equation}\label{c2}
	  w_{p}(\f_*\mu)\leq v.
	 \end{equation}
	 \rm{(iii)}\begin{equation}\label{c3}
	(\ref{covolume3}) \text{ holds for some } (\implies\text{ for any }) ~R \text{ satisfying } (\ref{R-type}).
	 \end{equation}
	 \rm{(iv)}  \begin{equation}\label{c5}
	(\ref{covolume1}) \text{ holds for some } (\implies\text{ for any }) ~R \text{ satisfying } (\ref{R-type}).
	 \end{equation}	 
\end{theorem}
	\begin{proof}
		
We have already observed that (\ref{c5}) holds if and only if so does (\ref{c3}),  
and that  (\ref{c2}) implies (\ref{c1}) by definition. It remains to show that (\ref{c3}) $\implies$ (\ref{c2}) and that (\ref{c1}) $\implies$ (\ref{c5}). 
Assume (\ref{c3}) and since $\mu$ is Federer and $(\f, \mu)$ is good, we have for $\mu$-a.e $x\in X$ has a neighbourhood $V$ such that $(\f, \mu)$ good and $\mu$ is $D$-federer on $V$. Choose a ball $B= B(x, r) $ of positive measure such that the dilated ball $\tilde B= B(x, 3^{n+1}r)$ is contained in $V$.  Since we have already noticed  in (\ref{covolume1}) that $\cov(g_tu_{\f(x)}\Delta) $ is $\max$ of norm of linear combinations of $1, f_1, \cdots, f_n$ condition (i)of (\ref{prop1}) is satisfied. And (\ref{c3}) is equivalent to second hypothesis (\ref{A2}) in (\ref{prop1}). Therefore we can conclude (\ref{c2}).\\
		\noindent
		Suppose (\ref{c5}) does not hold then equivalently condition (\ref{A2}) does not hold and by Lemma \ref{necessary} it follows that $$f( B\cap\supp\mu)\subset \mathcal{W}^{p}_u$$ for some $u>v$.
	\end{proof}
The upshot of the last theorem is that (\ref{c3}) and (\ref{c5}) do not involve any of $\f, \mu, X$. So if conditions (\ref{c1}), (\ref{c2}) hold for some $\f, \mu, X$ satisfying the hypotheses of the Theorem then they hold any other $\f^\prime, \mu^\prime, X^\prime$ satisfying the hypotheses of this theorem. So for any two $\f, \mu, X$ and  $\f^\prime, \mu^\prime, X^\prime$ satisfying the hypotheses of this theorem, we have that $w_p(\f_*\mu) = w_p(\f^\prime_*\mu^\prime)$.

\begin{theorem}
Let $\mu$ be a Federer measure on a Besicovitch metric space $X, \mathcal{L}$ an affine subspace of $\Q_p^n$, and let $\f : X\to \mathcal{L}$ be continuous map such that $(\f, \mu)$ is good and nonplanar in $\mathcal{L}$. Then 
\begin{equation}
w_{p}(\f_*\mu)=w_{p}(\mathcal{L})= \inf\{ w_{p}(\by)~|~\by\in\mathcal{L}\}= \inf\{w_{p}(\f(x)) ~|~ x\in\supp\mu\}.
\end{equation}
\end{theorem}
\begin{proof}
 Let us take $\nu= \lambda$ (Haar measure) on $\Q_p^s$ and $\mathbf{h}$ as in (\ref{R-type}). Then by definition, $w(\mathcal{L})= w_{p}(\mathbf{h}_*\nu)$. Since $\mathbf{h}, \nu$ and $\Q_p^s$ satisfy the hypotheses of Theorem \ref{main theorem}, we have $w_{p}(\f_*\mu) = w_{p}(\mathcal{L})$ by the previous discussion. And we already have from definitions that 
 $$  \inf\{ w_{p}(\by)~|~\by\in\mathcal{L}\}\leq \inf\{w_{p}(\f(x)) ~|~ x\in\supp\mu\}\leq w_{p}(\f_*\mu).$$
 Therefore to conclude the theorem it is enough to show that $$w_p(\mathcal{L})\leq \inf\{ w_p(\by)~|~\by\in\mathcal{L}\}.$$ 
 But this follows by taking $v= \inf\{ w_p(\by)~|~\by\in\mathcal{L}\}$ and $\mathbf{h}$ in Theorem \ref{main theorem}; condition (\ref{c1}) automatically holds. 
\end{proof}
\begin{corollary}\label{w(L)}
	Let $\cL$ be an $s$-dimensional affine subspace of $\Q_p^n$. Then $$ w_{p}(\mathcal{L}) =\max(n, \inf\{v \text{ for which (\ref{covolume3})  holds for } R \text{ as in } (\ref{R-type})\}$$
\end{corollary}
In view of Proposition $3.1$, we have
\begin{corollary}\label{main_theorem}
	Let $\mu$ be a Federer measure on a Besicovitch metric space $X, \mathcal{L}$ an affine subspace of $\Q_p^n$, and let $\f : X\to \mathcal{L}$ be continuous map such that $(\f, \mu)$ is good and nonplanar in $\mathcal{L}$. Then 
	\begin{equation}
	w(\f_*\mu)=w(\mathcal{L})= \inf\{ w(\by)~|~\by\in\mathcal{L}\}= \inf\{w(\f(x)) ~|~ x\in\supp\mu\}.
	\end{equation}
\end{corollary}

\section{Computation of Higher Diophantine Exponents}\label{higher}
In this section we want to  relate condition (\ref{covolume3}) in terms of Diophantine conditions of a parametrizing matrix of the $s$ dimensional subspace $\cL$. Set \begin{equation}\label{A} R= R_A :=(I_{s+1}~ ~ A)\in M_{s+1,n+1}\end{equation} where $A=(a_{ij})_{\substack{ i= 0,\cdots s\\ j= s+1,\cdots,n}}=\begin{bmatrix}
\ba_0\\ \vdots\\ \ba_s
\end{bmatrix}\in M_{s+1, n-s}$. Then we can write $$ \Vert R_A \bc(\bw)\Vert_p=\max_{i= 0,\cdots, s}\max\limits_{\substack{J\subset\{1,\cdots,n\} \\\# J=j-1}}\vert \langle (\be_i^p+\ba_i)\wedge \be_J^p, \bw\rangle\vert_p.$$
Let $\pi_{\bullet}$ be the projection from $\bigwedge^j \Q_p^{n+1}\bigoplus\bigwedge^j\R^{n+1}$ to 
$$\bigwedge^j \Q_p\langle \be^p_{s+1}, \cdots,\be^p_{n}\rangle\bigoplus\bigwedge^j \R\langle\be^\infty_{s+1}, \cdots,\be^\infty_{n}\rangle$$ 
where $F\langle \be_{s+1}, \cdots,\be_{n}\rangle$ denotes the $F$ span of $\be_{s+1}, \cdots, \be_n$, the last $n-s$ vectors of standard basis of $F^{n+1}$ for a field $F$.
\begin{lemma}\label{pidot}
		Suppose that $\Vert R_A \bc(\bw)\Vert_p \Vert \bw\Vert_\infty \leq 1$ for some $\bw \in\mathcal{S}_{n+1,j}$. Then $\Vert\bw\Vert_p\Vert\bw\Vert_\infty\ll 1+\Vert\pi_{\bullet}(\bw)\Vert_p\Vert\bw\Vert_\infty$.
	\end{lemma}
\begin{proof}
	 Suppose the smallest index of $I\subset \{1, \cdots, n\}( \# I=j )$ is $m$. If $m>s$ then we have that $$\vert \langle \be_I, \bw\rangle\vert_p\leq \Vert\pi_{\bullet}(\bw)\Vert_p.$$ On the other hand, when $m\leq s$ we have 
	 $$\begin{aligned}
	 	&\vert \langle \be_I^p, \bw\rangle\vert_p\Vert\bw\Vert_\infty \\
	 	&< \vert \langle (\be^p_m+\ba_m)\wedge\be^p_{I\setminus\{m\}}, \bw\rangle\vert_p\Vert\bw\Vert_\infty+ \vert \langle \ba_m\wedge\be^p_{I\setminus\{m\}}, \bw\rangle\vert_p\Vert\bw\Vert_\infty\\
	 	& <\Vert R_A \bc(\bw)\Vert_p \Vert \bw\Vert_\infty+ \max_{i=0, \cdots, s}\Vert\ba_i\Vert_p\max_{i= 0,\cdots, s}\max\limits_{\substack{J\subset\{m+1,\cdots , n\} \\ \# J=j}}\vert \langle \be^p_J, \bw\rangle\vert_p\Vert\bw\Vert_\infty\\ 
	 	&< 1+ \max_{i=0, \cdots, s}\Vert\ba_i\Vert_p\max_{i= 0,\cdots, s}\max\limits_{\substack{J\subset\{m+1,\cdots , n\} \\ \# J=j}}\vert \langle \be^p_J, \bw\rangle\vert_p\Vert\bw\Vert_\infty.
	 	\end{aligned}$$
	Now the same argument can be applied to the each of the components $\vert \langle \be^p_J, \bw\rangle\vert_p\Vert\bw\Vert_\infty$ where the smallest index is at least $m+1$. After at most $s+1$ steps the process terminates and this concludes the proof. 
	\end{proof}

Note that for $\bw\in\mathcal{S}_{n+1, j}$ with $j > n-s$ we have that $\pi_{\bullet}(\bw)=0$. Since $$\Vert\pi(\bw)\Vert_p\Vert\bw\Vert_\infty\geq 1 ~\forall ~\bw\in\mathcal{S}_{n+1, j},$$ we have that   
$$\Vert R_A \bc(\bw)\Vert_p \Vert \bw\Vert_\infty \geq 1\geq (\Vert\pi(\bw)\Vert_p\Vert\bw\Vert_\infty)^{-u}$$ 
for $u > 0$ and for all arbitrarily large $\Vert\pi(\bw)\Vert_p\Vert\bw\Vert_\infty$. Thus condition (\ref{covolume3}) holds for $j>n-s$. Therefore while computing the exponent $w_{p}(\mathcal{L}),$ according to Corollary \ref{w(L)}, the subgroups of rank greater than $n-s$ are irrelevant. We now define higher Diophantine exponents. 
\begin{definition}\label{def:higher}
	For each $j= 1, \cdots, n-s,$ let $$
	w^{p}_{j}(A):= \sup \left\{v\left| \begin{aligned} &\exists ~\bw\in \mathcal{S}_{n+1,j}\text{ with arbitrary large }\Vert\pi_{\bullet}(\bw)\Vert_p\Vert\bw\Vert_\infty\\& \text{such that }\Vert R_A\bc(\bw)\Vert_p\Vert\bw\Vert_\infty<(\Vert\pi_{\bullet}(\bw)\Vert_p\Vert\bw\Vert_\infty)^{-\frac{v+1-j}{j} }\end{aligned}\right. \right\}.$$
\end{definition}
We now observe 
\begin{corollary}\label{max w_j}
	For an s-dimensional subspace $\mathcal{L}$ parametrized by $A$ as in (\ref{A})
	 $$
	w_{p}(\mathcal{L})= \max (n, w^{p}_j(A)_{j=1,\cdots,n-s}).
	$$
\end{corollary}
\begin{proof}
By Lemma \ref{pidot} we can conclude that for some $v$ condition (\ref{covolume3}) holds if and only if $\max_{j=0,\cdots,n-s} w^{p}_j(A)\leq v $. And thus by (\ref{w(L)}), we can conclude this corollary.
\end{proof}
\begin{lemma}\label{w=w_1}
	$w^{p}_1(A)= w_{p}(A).$
\end{lemma}

\begin{proof}
	Take $\bw=\begin{bmatrix}
	q_0,\\ q_1\\ \vdots\\q_s\\q_{s+1}\\ \vdots\\q_n
	\end{bmatrix}=\tilde{\bq}\in\mathcal{D}^{n+1}\setminus\{0\}= \mathcal{S}_{n+1,1}$ and denote $\bq_0=\begin{bmatrix}
	q_0\\\vdots\\q_s
	\end{bmatrix}$ and $\bq= \begin{bmatrix}
	q_{s+1}\\\vdots\\q_{n}
	\end{bmatrix}$. Then $$R_A(\bc(\bw))=\bq_0+ A\bq , \bc(\bw)=\bw \text{ and } \pi_{\bullet}(\bw)=\bq.$$
	Hence the definition of both exponents coincide.
\end{proof}
\begin{theorem}\label{hyperplane}
	Suppose $\mathcal{L}$ is an $s= n-1$ dimensional subspace parametrized by $R_A$ as in (\ref{A}). Then $$w_{p}(\mathcal{L})=\max(n, w_{p}(A)).$$
\end{theorem}
\begin{proof}
	The Theorem follows directly from Lemma \ref{w=w_1} and Corollary \ref{max w_j}.
\end{proof}

\begin{lemma}
	Let $\mathcal{L}$ be parametrized  by $R_A$ as in (\ref{A}). Then for any $v< w_{p}(A)$ there exists arbitrary large $\Vert\bq\Vert_p\Vert\tilde{\bq}\Vert_\infty $ such that \begin{equation}
	\vert \bq.\by+q_0\vert_p\Vert\tilde{\bq}\Vert_\infty\leq (\Vert \bq\Vert_p \Vert\tilde{\bq}\Vert_\infty)^{-v} \text{ for all } \by\in\mathcal{L} \text{ and } \tilde{\bq}=(q_0,\bq)\in\mathcal{D}^{n+1}.\end{equation}
\end{lemma}

\begin{proof}
	Note that having arbitrarily large $\Vert\bq\Vert_p\Vert\tilde{\bq}\Vert_\infty$ is not the same as infinitely many $\tilde{\bq}\in\mathcal{D}^{n+1}$.
	Since $v< w_{p}(A)$ there exists $\varepsilon>0$ such that $$v<v+2\varepsilon< w_{p}(A).$$  Hence there exists arbitrary large $\Vert\bq\Vert_p\Vert\tilde{\bq}\Vert_\infty $ such that $$\Vert\bq_0+ A\bq\Vert_p\Vert \tilde\bq\Vert_\infty\leq (\Vert\bq\Vert_p\Vert\tilde{\bq}\Vert_\infty)^{-v-2\varepsilon}$$ and $\tilde{\bq}=(\bq_0,\bq)=\left(\begin{bmatrix}
	q_0\\ \vdots\\q_s
	\end{bmatrix} ,\begin{bmatrix}
	q_{s+1}\\ \vdots \\ q_n
	\end{bmatrix}\right)\in\mathcal{D}^{s+1}\times\mathcal{D}^{n-s}$.\\ For any $\by=(\bx,\tilde\bx A)\in\mathcal{L}$ where $\tilde \bx=(1,\bx)\in\Q_p^{s+1}$ we can write $$\begin{aligned}\left\vert q_0+\by.\begin{bmatrix}
	q_1\\\vdots\\ q_n
	\end{bmatrix}\right\vert_p\Vert\tilde{\bq}\Vert_\infty
&=\left\vert q_0+(\bx, \tilde{\bx}A).\begin{bmatrix}
q_1\\\vdots\\ q_n
\end{bmatrix}\right\vert_p\Vert\tilde{\bq}\Vert_\infty\\
&=\left\vert q_0+\bx.\begin{bmatrix}
q_1\\ \vdots \\q_s
\end{bmatrix}+\tilde{\bx}A.\begin{bmatrix}
q_{s+1}\\ \vdots\\ q_n
\end{bmatrix}\right\vert_p\Vert\tilde{\bq}\Vert_\infty\\& =\vert\tilde{\bx}(A\bq+ \bq_0)\vert_p\Vert\tilde\bq\Vert_\infty\\&\leq \Vert\tilde{\bx}\Vert_p\Vert A\bq+ \bq_0\Vert_p\Vert\tilde\bq\Vert_\infty\\&
\leq \Vert\tilde{\bx}\Vert_p (\Vert \bq\Vert_p \Vert\tilde{\bq}\Vert_\infty)^{-v-2\varepsilon} 	\end{aligned}
	$$ for arbitrary large $\Vert\bq\Vert_p\Vert\tilde{\bq}\Vert_\infty $. This implies that we also have $$\left\vert q_0+\by.\bq_{1}\right
	\vert_p\Vert\tilde{\bq}\Vert_\infty
	\leq (\Vert \bq\Vert_p \Vert\tilde{\bq}\Vert_\infty)^{-v-\varepsilon}.
$$  for arbitrary large $\Vert\bq\Vert_p\Vert\tilde{\bq}\Vert_\infty $ where $\bq_1= \begin{bmatrix}
q_1\\ \vdots\\ q_n\end{bmatrix}$.\\
We also have $$\Vert\bq_0\Vert_p\Vert\tilde{\bq}\Vert_\infty\leq\Vert\bq_0\Vert_p\Vert\tilde{\bq}\Vert_\infty\leq (1+ \Vert A\Vert_p)\Vert\bq\Vert_p\Vert\tilde{\bq}\Vert_\infty$$ which implies that $$\Vert\bq_1\Vert_p\leq (1+\Vert A\Vert_p) \Vert\bq\Vert_p.$$
Then  we have that 
$$\begin{aligned}
&\vert q_0+\by.\bq_1\vert_p\Vert\tilde{\bq}\Vert_\infty\\&\leq \frac{1}{(\Vert \bq_1\Vert_p \Vert\tilde{\bq}\Vert_\infty)^{v+\varepsilon}}.\frac{(\Vert \bq_1\Vert_p \Vert\tilde{\bq}\Vert_\infty)^{v+\varepsilon}}{(\Vert \bq\Vert_p \Vert\tilde{\bq}\Vert_\infty)^{v+\varepsilon}}\\
& \leq \frac{1}{(\Vert \bq_1\Vert_p \Vert\tilde{\bq}\Vert_\infty)^{v}}
\end{aligned}
$$ for arbitrary large $\Vert \bq\Vert_p\Vert\tilde{\bq}\Vert_\infty$, hence  for arbitrary large 
$\Vert \bq_1\Vert_p\Vert\tilde{\bq}\Vert_\infty$,
thereby concluding the proof. 
\end{proof}

An immediate consequence of above lemma is the following corollary.
\begin{corollary}
		Let $\mathcal{L}$ is parametrized  by $R_A$ as in (\ref{A}). Then $$w_{p}(A)\leq w_{p}(\mathcal{L}).$$
\end{corollary}
Another observation is when in the parametrizing matrix $A$ with more than one column and  all columns are rational multiple of one column then $w_{p}(A)=\infty$ and so is $w_{p}(\mathcal{L})$. Combining this observation with Theorem \ref{hyperplane} we can conclude the following:
\begin{theorem}\label{columns}
	Let $\mathcal{L}$ be an affine subspace parametrized  by $R_A$ as in (\ref{A}). If all the columns are rational multiples of one column then $$
	w_{p}(\mathcal{L})= \max(n, w_{p}(A)) \text{ and } w(\mathcal{L})=\max(n+1, w(A)).
	$$
\end{theorem}
The calculation of symmetries of higher order exponents is exactly the same as in case of $\R$ as discussed in \cite{Kleinbock-exponent}. So we will just state them in the $p$-adic setup.
\begin{lemma}For any $A\in M_{s+1,n-s}$ and all $\bw\in \mathcal S_{n+1,{j}}$, 
$2\le j \le n-s$, one has 
$$
\max_{i = 0,\dots,s} \max\limits_{\substack{J
\subset\{0,\dots,n\}\\ \# J = j-1}}\big|\big\langle (\be_i^p + \ba_i)\wedge\be_{ J}^p,\bw\big\rangle\big|_p \ll \|R_{ A}
\bc(\bw)\|_p
$$
and
$$
\|R_{A}
\bc(\bw)\|_p \ll \max_{i = 0,\dots,s} \max\limits_{\substack{J
\subset\{i+1,\dots,n\}\\ \# J = j-1}}
\big|\big\langle (\be_i^p + \ba_i)\wedge\be_{ J}^p,\bw\big\rangle\big|_p.
$$
\end{lemma}

The Lemma enables us to conclude that $w_j^{p}(A)$ is symmetric under any row operation. 
The next lemma allows to consider other row operations namely the multiplication by nonzero rationals and adding one row to another and transposition of rows.
 \begin{lemma}
 Let $A' = BA$ for some $B\in \GL_{s+1}(\Q)$; in other words, $A'$ can 
be obtained from $A$  by a sequence of elementary row operations with rational coefficients.
 Then $w^{p}_j(A') = w^{p}_j(A)$ for all $j$. 
\end{lemma}

Now we will wrap up this section by stating the last lemma and a theorem similar to Theorem \ref{columns}.
for rational multiples of rows.
\begin{lemma}
	 Suppose that $A$ has more than one row, and 
	let $A'$ be the matrix obtained from $A$ by removing one of its
	rows. Then  $w_j^{p}(A') \ge w_j^{p}(A)$ for all $j$.
	If in addition the removed row  is a rational linear combination
	of the remaining rows, then  $w_j^{p}(A') = w_j^{p}(A)$ for all $j$.
	\end{lemma}
As a consequence we have the following theorem.
\begin{theorem}
	Let $\mathcal{L}$ be an affine subspace parametrized by $R_A$ as in (\ref{A}). If all the rows are rational multiples of one row then one has $$
w_{p}(\mathcal{L})= \max(n, w_{p}(A)) \text{ and } w(\mathcal{L})=\max(n+1, w(A)).
$$
\end{theorem}

The lemmata and Theorem stated above are proved by a verbatim repetition of the arguments of Kleinbock \cite{Kleinbock-exponent}, and so we omit the proofs.



\begin{thebibliography}{99}
\bibitem{BBK} V.V. Beresnevich, V.I. Bernik, E.I. Kovalevskaya, \textit{On approximation of p-adic numbers by p-adic algebraic numbers}, Journal of Number Theory 111 (2005), 33--56.
\bibitem{BBKM} V. Beresnevich, V. Bernik, D. Kleinbock and G. Margulis, \textit{Metric Diophantine approximation~:~the Khintchine-Groshev theorem for non-degenerate manifolds}, Moscow Mathematical Journal 2:2 (2002), 203--225.
\bibitem{BGGV}  V. Beresnevich, A. Ganguly, A. Ghosh and S. Velani, \textit{Inhomogeneous dual Diophantine approximation on affine subspaces}, https://arxiv.org/abs/1711.08559, to appear in International Mathematics Research Notices.
\bibitem{BK} V.V. Beresnevich, E.I. Kovalevskaya, \textit{On Diophantine approximations of dependent quantities in the p-adic case}, Mat. Zametki 73:1 (2003), 22--37; translation: Math. Notes 73:1-2 (2003), 21--35.
\bibitem{Bugeaud-book}Y. Bugeaud, \textit{Approximation by Algebraic Numbers}, Cambridge Tracts in Mathematics, 160
(Cambridge University Press, Cambridge, 2004).
\bibitem{BBDO} Y. Bugeaud, Y. Budarina, H. Dickinson and H. O'Donnell, \textit{On simultaneous rational approximation to a p-adic number and its integral powers}, Proc. Edinb. Math. Soc. (2) 54 (2011), no. 3, 599--612. 
\bibitem{CaoYu} R. Cao and J. You, \textit{Diophantine vectors in analytic submanifolds of Euclidean spaces}, Sci. China Ser. A
50 (2007) 1334--8.
\bibitem{DG} Shreyasi Datta, Anish Ghosh, S-arithmetic Inhomogeneous Diophantine approximation on manifolds, arXiv:1801.08848v1
\bibitem{GG} A. Ganguly and A. Ghosh, \textit{Quantitative Diophantine approximation on affine subspaces}, https://arxiv.org/abs/1610.02157, to appear in Math. Z.
\bibitem{G1} A. Ghosh, \textit{A Khintchine-type theorem for hyperplanes}, J. London Math. Soc. \textbf{72}, No.2 (2005), pp. 293--304.
\bibitem{G-thesis} A. Ghosh, \textit{ Dynamics of homogeneous spaces and Diophantine approximation on manifolds}, Thesis (Ph.D.) Brandeis University. 2006. 55 pp. ISBN: 978-0542-56323-2.
\bibitem{Gpos} A. Ghosh, \textit{Metric Diophantine approximation over a local field of positive characteristic}, J. Number Theory 124 (2007), no. 2, 454--469.
\bibitem{G-div} A. Ghosh, \textit{A Khintchine Groshev Theorem for Affine hyperplanes}, International Journal of Number Theory, 7 (2011), no. 4, 1045--1064.
\bibitem{G-mult} A. Ghosh, \textit{Diophantine approximation on affine hyperplanes}, Acta Arithmetica, 144 (2010), 167--182.
\bibitem{G-monat} A. Ghosh, \textit{Diophantine approximation and the Khintchine-Groshev theorem}, Monatsh. Math. \textbf{163} (2011), no. 3, 281--299.
\bibitem{G-handbook} A. Ghosh, \textit{Diophantine approximation on subspaces of $\mathbb{R}^n$ and dynamics on homogeneous spaces},  Handbook of Group Actions (Vol. IV) ALM 41, Ch. 9, pp. 509--527.
\bibitem{GGN1} A. Ghosh,  A. Gorodnik and A. Nevo, \textit{Diophantine approximation and automorphic spectrum}, Int Math Res Notices (2013) 2013 (21): 5002--5058.
\bibitem{GGN2} A. Ghosh, A. Gorodnik and A. Nevo, \textit{Metric Diophantine approximation on homogeneous varieties}, Compositio Mathematica, 150 (2014), issue 08, 1435--1456.	
\bibitem{Kleinbock-extremal} D. Kleinbock, \textit{Extremal subspaces and their submanifolds}, Geom. Funct. Anal \textbf{13}, (2003), No 2, pp.437--466.
\bibitem{Kleinbock-exponent} D. Kleinbock, \textit{An extension of quantitative nondivergence and applications to Diophantine exponents}, Trans. Amer. Math. Soc. 360 (2008), no. 12, 6497--6523.

\bibitem{KLW} D. Kleinbock, E. Lindenstrauss and B. Weiss, \textit{On fractal measures and Diophantine approximation}, 
Selecta Math. (N.S.) 10 (2004), no. 4, 479--523.
\bibitem{KM} D. Kleinbock and G. A. Margulis, \textit{Flows on homogeneous spaces and Diophantine Approximation on Manifolds}, Ann Math\textbf{148}, (1998), pp.339--360.
\bibitem{KT} D. Kleinbock and G. Tomanov, \textit{Flows on $S$-arithmetic homogeneous spaces and applications to metric Diophantine approximation}, Comm. Math. Helv. 82 (2007), 519--581.
\bibitem{Lutz} E. Lutz, \textit{Sur les approximations diophantiennes lin\'{e}aires P-adiques}, Actualit\'{e}s Sci. Ind., no. 1224, Hermann \& Cie, Paris, 1955. (French)
\bibitem{Mah2} K. Mahler, \textit{\"{U}ber diophantische Aproximationen im Gebieteder $P$-adischen Zahlen}, Jber. Deutsch. Math. Verein. 44 (1934), 250--255.
\bibitem{Mah} K. Mahler, \textit{\"{U}ber eine Klasseneinteilung der $p$-adischen Zahlen}, Mathematica (Leiden) 3 (1935), 177--185.
\bibitem{Mosh} N. Moshchevitin, \textit{On Kleinbock's Diophantine result}
Publ. Math. Debrecen 79 (2011), no. 3-4, 531--537.
\bibitem{Sevryuk}  Mikhail B. Sevryuk, \textit{KAM tori: persistence and smoothness}, Nonlinearity 21 (2008), no. 10, T177--T185.
\bibitem{Sp}V. G. Sprind\v{z}uk, \textit{Achievements and problems in Diophantine Approximation theory}, Russian Math. Surveys \textbf{35} (1980), pp. 1--80.
\bibitem{Sp3} V. G. Sprind\v{z}uk, \textit{Metric theory of Diophantine approximations}, John Wiley \& Sons, New York-Toronto-London, 1979.
\bibitem{Sp2} V. G. Sprind\v{z}uk, \textit{Mahler's problem in metric number theory}, Translated from the Russian by B. Volkmann. Translations of Mathematical Monographs, Vol. 25 American Mathematical Society, Providence, R.I. 1969 vii+192 pp.
\bibitem{SuRafa} X. Su and R. de la Llave, \textit{KAM theory for quasi-periodic equilibria in one-dimensional quasi-periodic media}, SIAM J. Math. Anal. 44 (2012), no. 6, 3901--3927.

		
	\end{thebibliography}
\end{document}